\documentclass[12pt,a4paper]{article}

\usepackage{amsmath, amssymb, enumerate, mathrsfs}
\usepackage[english]{babel}
\usepackage{amsfonts}
\usepackage{amsthm}
\usepackage[pagebackref=true]{hyperref}
\usepackage[alphabetic]{amsrefs}
\usepackage{geometry}
\usepackage{enumerate}
\usepackage{url}
\usepackage{xspace}
\usepackage[all]{xy}
\usepackage{graphicx}

\usepackage{authblk}

\usepackage{color}      

\theoremstyle{plain}
\newtheorem{thm}{Theorem}[section]
\newtheorem{lem}[thm]{Lemma}
\newtheorem{prop}[thm]{Proposition}
\newtheorem{cor}[thm]{Corollary}
\newtheorem*{claim*}{Claim}
\newtheorem{e-definition}[thm]{Definition\rm}
\newtheorem{ques}[thm]{Question}

\theoremstyle{definition}
\newtheorem{example}[thm]{Example}
\newtheorem{defn}[thm]{Definition}

\newtheorem*{rem*}{\it Remark\/}

\newcommand{\FF}{\mathbf{F}}

\newcommand{\ZZ}{{\mathbf{Z}}}
\newcommand{\NN}{{\mathbf{N}}}

\newcommand{\RR}{{\mathbf{R}}}

\newcommand{\gd}{\delta}

\newcommand{\gO}{\Omega}

\newcommand{\gep}{\varepsilon}

\newcommand{\vareps}{\varepsilon}

\newcommand{\LF}{\mathrm{Rad_{LE}}}
\newcommand{\LFp}{\mathrm{Rad}_p}

\newcommand{\tdlc}{t.d.l.c.\xspace}

\newcommand{\ti}[1]{\tilde{#1}}

\newcommand{\vol}{\mathrm{vol}}

\newcommand{\SL}{\mathrm{SL}}
\newcommand{\GL}{\mathrm{GL}}

\newcommand{\Aut}{\mathrm{Aut}}

\newcommand{\norma}{\mathscr{N}}

\newcommand{\inv}{^{-1}}

\newcommand{\Ker}{\mathrm{Ker}}

\def\og{\leavevmode\raise.3ex\hbox{$\scriptscriptstyle\langle\!\langle$~}}
\def\fg{\leavevmode\raise.3ex\hbox{~$\!\scriptscriptstyle\,\rangle\!\rangle$}}

\begin{document}

\title{Lattices in amenable groups}

\author[1]{Uri Bader\thanks{Supported in part by the ISF-Moked grant 2095/15 and the ERC
grant 306706.}}

\author[2]{Pierre-Emmanuel Caprace\thanks{F.R.S.-FNRS research associate, supported in part by the ERC grant 278469}}
\author[1]{Tsachik Gelander\thanks{Supported in part by the ISF-Moked grant 2095/15.}}

\author[3]{Shahar Mozes\thanks{Supported in part by the ISF-Moked grant 2095/15.}}

\affil[1]{Weizmann Institute, Rehovot}
\affil[2]{Universit\'e catholique de Louvain}
\affil[3]{Hebrew University of Jerusalem}

\date{Dedicated to the memory of George Daniel Mostow with friendship and admiration}

\selectlanguage{english}
\maketitle

\baselineskip=17pt

\begin{abstract}
\baselineskip=16pt
Let $G$ be a locally compact amenable group. We say that $G$ has \textbf{property (M)} if   every closed subgroup of finite covolume in $G$ is cocompact. A classical theorem of Mostow ensures that connected solvable Lie groups have property (M). We prove a non-Archimedean extension of Mostow's theorem by showing the amenable linear locally compact groups have property (M). However  property (M) does not hold for all solvable locally compact groups: indeed, we  exhibit an example of a metabelian locally compact group with a non-uniform lattice. We show that compactly generated metabelian groups, and more generally nilpotent-by-nilpotent groups, do have property (M).  Finally, we highlight a connection of property (M) with the subtle relation  between the analytic notions of strong ergodicity and the spectral gap.
\end{abstract}

AMS Mathematics Subject Classification: 22E40, 22D05, 22G25, 22F10, 22F30.

Keywords: lattices, discrete subgroups, amenable groups.

\tableofcontents

\section{Introduction}

The starting point of this work is the following classical result due to Mostow~\cite{Mostow}:

\begin{thm}[Mostow]\label{theo:Mostow}
Let $G$ be a solvable Lie group and $H$ be a closed subgroup. Then $G/H$ carries a $G$-invariant probability measure if and only if $G/H$ is compact.
\end{thm}

Mostow's original proof is quite involved, and uses in an essential way the structure theory of solvable Lie groups and their Lie algebras. A simplified (but still Lie theoretic) proof due to Raghunathan, may be found in \cite[Theorem~3.1]{Raghu}. One direction of the equivalence in Theorem~\ref{theo:Mostow} is actually straightforward to establish for general amenable groups without invoking any structure results on Lie groups: if $G$ is amenable then any compact $G$-space carries a $G$-invariant probability measure by the fixed point property of amenable groups. The present paper is devoted to the question whether the converse implication also holds for general amenable locally compact groups.

\begin{ques}\label{ques:v1}
Let $G$ be an amenable ($\sigma$-compact) locally compact group, and $H$ a closed subgroup.

Is it true that, if $H$ is of finite covolume, then it is cocompact?
\end{ques}

Our initial motivation in addressing that question came from an earlier work \cite{BCGM}, when we observed that a positive answer to Question~\ref{ques:v1} would significantly simplify the proof of the main result from loc.~cit.

We say that $H$ is \textbf{of finite covolume}, or shortly \textbf{cofinite}, if $H$ is closed and $G/H$ carries a $G$-invariant probability measure. We say that $H$ is \textbf{cocompact} if $H$ is closed and $G/H$ is compact.

\begin{defn}
Let $G$ be a locally compact group. We will say that $G$ has property $(M)$ if every cofinite subgroup of $G$ is cocompact.\footnote{We have chosen the letter $M$ since Mostow proved this property for solvable Lie groups after Malcev proved it for nilpotent Lie groups.}
\end{defn}

This paper focuses on the study of property $(M)$ for amenable groups.
In order to illustrate that, consider the most basic class of amenable groups, namely abelian groups. If $G$ is abelian, any subgroup is normal. In particular, if $H$ is a cofinite subgroup of $G$, then the quotient space $G/H$ is a locally compact group carrying a finite invariant measure. This measure is thus proportional to the Haar measure on $G/H$, and it follows that $G/H$ is compact since any non-compact group has infinite Haar measure. That argument can in fact be promoted to cover the case of all nilpotent locally compact groups, see Proposition \ref{prop:nilpotent}. Mostow's theorem mentioned above provides a positive answer to the question in case $G$ is a solvable Lie group. In fact, Mostow's (or Raghunathan's) proof  can easily be adapted to cover the case of all amenable Lie groups. More generally, one has the following reduction due to Benoist--Quint \cite{BenoistQuint}.

\begin{thm}[See Proposition 3.4 in \cite{BenoistQuint}]\label{theo:reduction}
Let $G$ be an amenable locally compact group, and $H$ be a closed subgroup of finite covolume. Then $H$ is cocompact in $\overline{ G^\circ H}$.

In particular $H$ is cocompact in $G$ if and only if the closure of the image of $H$ in the group of components $G/G^\circ$ is cocompact.
\end{thm}

If the identity component $G^\circ$ is open in
$G$, then $G/G^\circ$ is discrete so that any subgroup of finite covolume is obviously cocompact. This is in particular the case if $G$ is a Lie group, so  Mostow's theorem follows immediately from Theorem~\ref{theo:reduction}. The proof of Theorem~\ref{theo:reduction}  relies on deep structure theory: one uses the solution of   Hilbert's fifth problem to reduce to the case where $G^\circ$ is a Lie group. Structure theory of Lie groups ensures that a connected amenable Lie group is   solvable-by-compact; one is then in a position to argue by induction on the derived length of the solvable radical, as Mostow did in his original proof.

As soon as one leaves the realm of almost connected groups, the algebraic structure of amenable (and even solvable) groups becomes more intricate. In order to illustrate this, let us mention two examples illustrating  that, in contrast with the Lie group case, the intersection of a lattice in a compactly generated nilpotent (resp. solvable) locally compact group with the center (resp. the derived group) of the ambient group need not be a lattice.

\begin{example}
Let
$$H = \left\{ \left(\begin{array}{ccc} 1 & x & z \\
0 & 1 & y\\
0 & 0 & 1
\end{array} \right)
\mid x, y, z \in \ZZ\right\}$$
be the $3$-dimensional Heisenberg group over $\ZZ$, and let
$U $ be the upper unitriangular subgroup of compact group $\SL_4(\ZZ_p)$.
Then $G = H \times U$ is compactly generated and $3$-step nilpotent. Let $\varphi \colon H \to U$ be the group homomorphism defined by
$$\varphi(
\left(
\begin{array}{ccc}
1 & x & z \\
0 & 1 & y\\
0 & 0 & 1
\end{array} \right) )  =
\left(
\begin{array}{cccc}
1 & x & z & 0 \\
0 & 1 & y & 0\\
0 & 0 & 1 & 0\\
0 & 0 & 0 & 1
\end{array} \right)
$$
and let $\Gamma \leq H \times U = G$ be the graph of $\varphi$.
Then $\Gamma$ is a cocompact lattice in $G$. Moreover the center $Z(G)$ of $G$ is non-compact, but the intersection of $\Gamma \cap Z(G)$ is trivial, and is thus not a lattice in $Z(G)$.
\end{example}

\begin{example}
Let $U = \varprojlim C_2 \wr C_{2^n}$ be the profinite wreath product of $C_2$ by $\ZZ_2 = \varprojlim C_{2^n}$, so that the lamplighter group $C_2 \wr \ZZ$  embeds as a dense subgroup of $U$. Let $G = D_\infty \times U$ be the direct product of the infinite dihedral group $D_\infty$ with $U$, so that $G$ is compactly generated and metabelian.  Let $t_1 \in D_\infty$ be a generator of the infinite cyclic subgroup of index~$2$, and $t_2 \in U$ be a topological generator of the semi-direct factor $\ZZ_2$ of $U$. Then the infinite cyclic group $\Gamma = \langle (t_1, t_2) \rangle \leq G$ is   a cocompact lattice in $G$. The closed derived group $\overline{[G, G]}$ is isomorphic to the direct product $\ZZ \times (\prod_\ZZ C_2)$, and intersects $\Gamma$ trivially.
\end{example}

The locally compact groups in both examples above are actually linear. Therefore they must have property (M) by virtue of the following result, which provides a positive answer to Question~\ref{ques:v1} for linear groups and may be viewed as   a non-Archimedean version of Mostow's theorem.

\begin{thm}\label{theo:LinearAmenable}
Let $n\geq 1$ be an integer, and for $i=1, \dots, n$, let  $G_i$ be an amenable locally compact  group. Assume that for each $i$, the group $G_i$ has a continuous faithful representation $G_i \to \GL_{d_i}(k_i)$ over a locally compact field $k_i$ (whose image is not required to be closed). Then the direct product $G = \prod_{i=1}^n G_i$ has property (M).
\end{thm}

The proof of Theorem~\ref{theo:LinearAmenable} relies on a description of the algebraic structure of amenable linear locally compact groups which is of independent interest, see Theorem~\ref{thm:StructureAmenableLinear} below.

As of today, there is however no structure theory encompassing \emph{all} amenable locally compact groups. In trying to exploit the hypothesis of amenability in the question above, one is thus naturally led to looking for a more analytic formulation of the question. In order to achieve this, we need to recall some ergodic theoretic definitions.

A measure-preserving action of a group $G$ on a probability space $(X, m)$ is said to have the \textbf{spectral gap} if the unitary representation of $G$ on the space $L^2_0(X, m)$ of (equivalence classes of) $L^2$-maps of zero integral does not contain almost invariant vectors. Following Margulis' terminology, a closed subgroup $H$ of finite covolume in a $\sigma$-compact locally compact group $G$ such that $G/H$ has the spectral gap, is called \textbf{weakly cocompact}. This is motivated by the fact, due to Margulis \cite[Corollary~III.1.10]{Margulis}, that if $H$ is cocompact, then it is automatically weakly cocompact. It turns out that if the ambient group is amenable, then the converse holds.

\begin{prop}\label{propo:WC}
Let $G$ be an amenable  second countable  locally compact group, and $H$ be a closed subgroup of finite covolume. Then
$H$ is cocompact if and only if $H$ is weakly cocompact.
\end{prop}

A closely related property of a probability measure preserving (p.m.p.) $G$-action on a space  $(X, m)$ is \textbf{strong ergodicity}: by definition, this is the property that   for every sequence $(B_n)$ of measurable subsets of $X$ which is almost invariant in the sense that $\lim_n m(B_n \Delta gB_n) = 0$ for all $g \in G$, one has $\lim_n m(B_n) (1 - m(B_n))=0$.  Strong ergodicity was introduced by Klaus Schmidt \cite{Schmidt80} as an invariant under orbit equivalence for discrete groups.  The following shows that strong ergodicity also pops up naturally in the context of non-discrete locally compact groups.

\begin{prop}\label{propo:SE}
Let $G$ be a second countable locally compact group and $H$ be a closed subgroup of finite covolume. Then the $G$-action on $G/H$ is strongly ergodic.

In other words, every continuous, transitive, p.m.p. action of a locally compact group is strongly ergodic.
\end{prop}

We shall present a short direct proof of Proposition~\ref{propo:SE}. Adrian Ioana has pointed out to us that this proposition can also be derived as a consequence of \cite[Lemma~2.5]{Ioana}. 

It is well known, and easy to see, that if a p.m.p. $G$-action has the spectral gap, then it is strongly ergodic.
The converse does not hold in general, as shown by Schmidt \cite[Example 2.7]{Schmidt81}. With Proposition~\ref{propo:SE} at hand, one also sees that if $\Gamma$ is a lattice in a   locally compact group $G$ which is not weakly cocompact, then the $G$-action on $G/\Gamma$ is strongly ergodic without the spectral gap. The first examples of such lattices were constructed in the automorphism group of a regular locally finite tree by Bekka--Lubotzky~\cite{BekkaLubo}, thereby providing non-discrete examples of strongly ergodic actions without the spectral gap.

However, the difference between strong ergodicity and the spectral gap for p.m.p. actions   is rather subtle. In fact, there are significant classes of groups for which both properties are equivalent for all p.m.p. actions. This is for instance trivially the case for locally compact groups with Kazhdan's property (T);  moreover, it is also the case for the class of countable amenable (discrete) groups, see  \cite[Theorem~2.4]{Schmidt81}. Using Propositions~\ref{propo:WC} and~\ref{propo:SE}, we can now reformulate Question~\ref{ques:v1} as follows.

\begin{ques}\label{ques:v2}
Let $G$ be an amenable second countable locally compact group with a measure-preserving action on a standard probability space $(X, m)$.

Is it true that if the $G$-action is strongly ergodic, then it has the spectral gap?
\end{ques}

It turns out that both Questions~\ref{ques:v1} and~\ref{ques:v2} have a negative answer. In order to describe a concrete example illustrating this matter of fact, consider a sequence of prime powers  $(q_n)_{n \geq 0}$. Let $\FF_{q_n}$ be the finite field of order $q_n$; we denote its multiplicative group by $\FF_{q_n}^*$. We define  the following groups:
$$
\Lambda = \bigoplus_{n \geq 0} \FF_{q_n} \hspace{1cm} \text{and} \hspace{1cm} S = \prod_{n \geq 0}  \FF_{q_n}^*.
$$
We endow $\Lambda$ with the discrete topology and $S$ with the product topology, so that $S$ is a compact group. Then $S$ acts continuously on $\Lambda$ by automorphisms: this amounts to saying that every finitely generated subgroup of $\Lambda$ is centralized by an open subgroup of $S$, which is indeed the case. We can thus form the semi-direct product
$$G = \Lambda \rtimes S,$$
which is a metabelian (in particular amenable), second countable, totally disconnected, locally compact group. It contains $S$ as a compact open subgroup and $\Lambda$ as a uniform (normal) lattice.

\begin{thm}\label{theo:example}
Assume that $\sum_n \frac 1 {q_n} < \infty$. Then the metabelian group $G$ does not have property $(M)$: Indeed $G$ contains a non-uniform lattice which, as an abstract group, is isomorphic to the direct sum  $\bigoplus_{n \geq 0}  \FF_{q_n}^*$.  Moreover $G$ contains uncountably many pairwise non-commensurable non-uniform lattices, and a single commensurability class of uniform lattices.

In particular, if $\Gamma$ is a non-uniform lattice in $G$, then the $G$-action on $G/\Gamma$ is strongly ergodic but does not have the spectral gap.
\end{thm}

Alain Valette pointed out to us that the same group $G$ as in Theorem~\ref{theo:example} appears in a completely different context in Alain Robert's paper \cite{Robert} (see also Chapter~16 in \cite{Robert_book}). That paper provides interesting information on the unitary representations of $G$; it does not consider  non-uniform lattices. 

By construction, the group $\Lambda$ is a discrete abelian cocompact normal subgroup of $G$. Clearly discrete groups have property (M), as do abelian groups. Hence Theorem~\ref{theo:example} illustrates   the fact that a locally compact group $G$ may fail to have (M) even if it has a cocompact closed normal subgroup with property (M). One may think that understanding exhaustively the lattices in a cocompact normal subgroup of a locally compact group $G$ is sufficient to understand all lattices in $G$. This is actually far from being true; Theorem~\ref{theo:example} provides an illustration of that fact\footnote{\baselineskip=16pt
Another illustration is provided by the group $G = \SL_3\big(\FF_p(\!(t)\!) \big) \rtimes \Aut\big(\FF_p(\!(t)\!) \big)$. While lattices in the cocompact normal subgroup $\SL_3\big(\FF_p(\!(t)\!) \big)$ are classified via the Arithmeticity theorem, it is a notorious and intriguing open problem to find a lattice in $G$ with infinite image in the field automorphism group, see \cite[Annexe A, Problem 1]{CorHab}. As pointed out to us by Yves de Cornulier, further illustrations may also be found in the realm of connected solvable Lie groups: there exist pairs $H\leq G$ of connected solvable Lie groups such that $H$ is a cocompact closed normal subgroup of $G$, and $G$ contains lattices while $H$ does not contain any.}.

Despite of that negative result, it seems  reasonable to expect  that   Questions~\ref{ques:v1} and~\ref{ques:v2} have a positive answer provided the group $G$ satisfies suitable (and natural) additional conditions. We shall obtain a variety of results in that direction in \S\ref{sec:PositiveResults} below. Among those, we mention   the following.

\begin{thm}\label{theo:nil-by-nil}
Let $G$ be a locally compact group with a closed normal subgroup $N$ such that $N$ and $G/N$ are both nilpotent.

If $G$ is compactly generated, then $G$ has property (M).
\end{thm}

In particular, compactly generated metabelian locally compact groups have property (M).  The hypothesis of compact generation is critical: the example from Theorem~\ref{theo:example} shows that the conclusion of the theorem may fail if one removes it.



\section{Strong ergodicity and spectral gap}

The goal of this section is to prove Proposition~\ref{propo:WC} and Proposition~\ref{propo:SE}.

\subsection{Strong ergodicity of transitive actions}\label{sec:SE}

Let $G$ be a locally compact group and $H$ be a closed subgroup and $\mu_H$ be a left Haar mesure on $H$. We recall from \cite[Corollary B.1.7]{BHV} that if $G/H$ carries some $G$-invariant measure $m$, then that measure is unique up to a scaling factor. Once $m$ is fixed, there is a unique left Haar measure $\mu _G$ on $G$ such that 
	$$\int_G f(x) d\mu_G(x) = \int_{G/H} \int_H f(xh)d\mu_H dm(xH)$$
	for every continuous compactly supported function $f$ on $G$. Similarly, if $\mu_G$ is fixed then there is a unique choice of $m$ so that the equation above holds. If this is the case, we shall say that the measures $\mu_G, \mu_H$ and $m$ satisfy the \textbf{standard normalization}.

An ergodic action on a finite probability space is always strongly ergodic, so we may ignore this case. For an infinite probability space $\Omega$
it is easy to see that an ergodic action is strongly ergodic if and only if there is a compact subset $K\subset G$ and $\gep>0$ with respect to which there is no $(K,\gep)$-invariant subsets of $\gO$ of measure $1/2$.


\begin{lem} \label{lem:haarnorm}
	Let $G$ be a locally compact group and $H\le G$ a cofinite subgroup. Let $\mu_G$ and $\mu_H$ be left Haar measures on $G$ and $H$.  
	Denote by $x_0\in G/H$ the base point, so that $H=\mathrm{Stab}_G(x_0)$.
	Let $m$ denote the $G$-invariant probability measure on $G/H$, and assume that $\mu_G, \mu_H$ and $m$ satisfy the standard normalization.  Then 
	we have
	$$
	\mu_G(K)=\int_{K\cdot x_0}\mu_H( H \cap x^{-1}K)dm(xH).
	$$
\end{lem}

\begin{proof}
By the choice of the Haar measure $\mu_G$, we have 
$$\begin{array}{rcl}
\mu_G(K) = \int_G \mathcal X_K(x) d\mu_G(x) 
&=& \int_{G/H} \int_H \mathcal X_K(xh)d\mu_H dm(xH)\\
& = & \int_{G/H}\mu_H(\{h \in H \mid h \in x^{-1}K\})dm(xH)\\
& = & \int_{G/H}\mu_H( H \cap x^{-1}K)dm(xH)\\
& = & \int_{K\cdot x_0}\mu_H( H \cap x^{-1}K)dm(xH)
\end{array}$$
as required. 
\end{proof}

\begin{proof}[Proof of Proposition \ref{propo:SE}]
	Let $m$ denote the $G$-invariant probability measure on $\Omega = G/H$. 
Choose a compact identity neighbourhood $O\subset H$ and let $\mu_H$ be the left Haar measure on $H$ such that $\mu_H(O)=1$.
Let $\mu_G$ be the left Haar measure on $G$ such that $\mu_G, \mu_H$ and $m$ satisfy the standard normalization.

	Denote by $x_0\in G/H$ the base point, so that $H=\mathrm{Stab}_G(x_0)$,
and fix a compact set $F\subset G$ such that the measure $m$ of $\mathcal{F}:=F\cdot x_0$ is at least $0.9$.
Set 
	$$
	K:=FOF^{-1}.
	$$
Let $\delta=\text{min}\{ \Delta_G(f^{-1}):~f\in F\}$ where $\Delta_G$ is the modular character of $G$.
We claim that for any $y \in \mathcal F$ and any measurable set $A\subseteq \mathcal{F}$, we  have
	$$
	\mu_G(\{k\in K \mid k\cdot y\in A\})\ge \gd\cdot m(A).
	$$
In order to prove the claim, it suffices to deal with the case where $A$ is compact, since a general $A$ can be approximated by compact sets by regularity of the measure. So we assume that $A$ is compact. We choose $f \in F$ with $f\cdot x_0 = y$. We have 
$$\{k\in K \mid k\cdot y\in A\} \supseteq \{q \in FO \mid q\cdot x_0 \in A\}f^{-1}.$$
Therefore it suffices to show that $\mu_G( \{q \in FO \mid q\cdot x_0 \in A\} ) \geq m(A)$.
By Lemma~\ref{lem:haarnorm}, we have 
$$\begin{array}{rcl}
\mu_G(\{q \in FO \mid q\cdot x_0 \in A\}) & = &  \int_A \mu_H(H \cap x^{-1} FO) dm(xH)\\
& \geq & \int_A dm = m(A)
\end{array}
$$
since $H \cap x^{-1} FO \supseteq O$ for all $x \in A \subseteq F$. This proves the claim.

Suppose by way of contradiction that there is $B=B_\gep\subset G/H$ of measure $1/2$ which is $(K,\gep)$-invariant for some positive  $\gep < 0.36\cdot\gd/\mu_G(K)$. Note that both $m(\mathcal{F}\cap B)$ and $m(\mathcal{F}\setminus B)$ are at least $0.4$. Using the claim, it follows that
	$$
	0.36\cdot\gd=0.9\cdot 0.4\cdot\gd\le
	\int_\mathcal{F}\int_K |\mathcal{X}_B(k\cdot x)-\mathcal{X}_B(x)|d\mu_G(k)dm(x)\le
	$$
	$$
	\int_{\gO}\int_K |\mathcal{X}_B(k\cdot x)-\mathcal{X}_B(x)|d\mu_G(k)dm(x)=
	$$
	$$
	\int_K\int_{\gO} |\mathcal{X}_B(k\cdot x)-\mathcal{X}_B(x)|dm(x)d\mu_G(k)\le\gep\mu_G(K).
	$$
This contradicts the fact that $B$ is $(K, \varepsilon)$-invariant.
\end{proof}

\subsection{Spectral gap and weak cocompactness}

%
For general $\sigma$-compact groups, Margulis showed that a cocompact lattice in $G$ is always weakly cocompact. More generally, he proved the following (see \cite[Lemma~III.1.9]{Margulis}).

\begin{lem}\label{lem:Margulis}
Let $G$ be a second countable locally compact group. A closed subgroup $H < G$ of finite covolume is not weakly cocompact if and only if there is a sequence $(q_i)_{i \geq0}$ of almost invariant vectors in $L_2(G/H)$ such that $\inf_i \| q_i \|_2 >0$ and for each compact subset $K \subset G/H$, we have $lim_{i \to \infty} \int_K | q_i(x) |^2  dx = 0$. \qed
\end{lem}

%

Thus, to complete the proof of Proposition~\ref{propo:WC}, it suffices to establish the following.

\begin{prop}\label{prop:SpectralGap}
Let $G$ be a second countable locally compact group, and $X$ be a $\sigma$-compact locally compact space. Assume that $G$ acts continuously on $X$, fixing a probability measure $m$ of full support in $X$.

If $G$ is amenable and $X$ is not compact, then the unitary representation of $G$ on $L_2^0(X, m)$ does not have the spectral gap.
\end{prop}

\begin{proof}
Let $(K_n)$ be a compact exhaustion of $G$, let $(\vareps_n)$ be a decreasing sequence of positive reals tending to $0$, and $(F_n)$ be a sequence of $(K_n, \vareps_n)$-invariant compact subsets of $G$ witnessing the amenability of $G$ (it is always possible to find F\o lner sets that are compact, by truncating if necessary).

Let also $(X_n)$ be a compact exhaustion of $X$. For each $n$, the set $F_n^{-1} X_n$ is compact since the $G$-action is continuous. If $X$ is not compact, there is a compact set of positive measure $A_n$ contained in $X \setminus F_n^{-1} X_n$. We set $\psi_n \colon X \to \RR$ to be the convolution $F_n\ast \mathcal{X}_{A_n}$, where $\mathcal X_{A_n}$ denotes the characteristic function of $A_n$. In other words, we have  
$$
 \psi_n(x)= \int_{F_n} \mathcal{X}_{h\cdot A_n}(x)d \mu (h), 
$$
where $\mu$ denotes a left Haar measure on $G$. We have
$$
\| \psi_n \|_1 =\int_{X} \int_{F_n} \mathcal{X}_{h\cdot A_n}(x)d \mu (h)d m(x) = \int_{F_n}  m(h\cdot A_n)d \mu (h) = m(A_n) \mu (F_n).
$$
Moreover, the support of $\psi_n$ is contained in $F_n A_n$, which lies entirely in the complement of $X_n$ by the definition of $A_n$.

Furthermore, for $g \in K_n$ we find
$$
\begin{array}{rcl}
\|  \psi_n - g. \psi_n \|_1  & =&  \int_{X} \big  |\psi_n(x) - \psi_n(g\inv.x) \big  | d m(x)  \\
& = & \int_{X} \big| \int_{F_n} \mathcal{X}_{h\cdot A_n}(x) \mathrm  d \mu (h)  -  \int_{F_n} \mathcal{X}_{gh\cdot A_n}(x)d \mu (h) \big | d m(x) \\
&  =  &  \int_{X} \big| \int_{F_n} \mathcal{X}_{h\cdot A_n}(x) \mathrm  d \mu (h)  -  \int_{gF_n} \mathcal{X}_{h\cdot A_n}(x)d \mu (h) \big | d m(x) \\
&  =  &  \int_{X} \big| \int_{F_n - gF_n} \mathcal{X}_{h\cdot A_n}(x) \mathrm  d \mu (h)  -  \int_{gF_n- F_n} \mathcal{X}_{h\cdot A_n}(x)d \mu (h) \big | d m(x) \\
&  \leq  &  \int_{X}  \int_{F_n \Delta gF_n} \mathcal{X}_{h\cdot A_n}(x) \mathrm  d \mu (h) d m(x) \\
& = & \int_{F_n \Delta gF_n}   m(h\cdot A_n)   \mathrm  d \mu (h)\\
& = & m(A_n) \mu (F_n \Delta gF_n) \\
& < & \vareps_n m(A_n) \mu (F_n) \\
& = & \vareps_n \| \psi_n \|_1.
\end{array}
$$
Thus the  map $\frac {\psi_n} { \| \psi_n \|_1}$ has norm one and is $(K_n, \vareps_n)$-almost invariant for the $L_1$-norm. Therefore, by \cite[Th.~9.1]{BL}, there is another sequence $(\vareps'_n)$ of positive reals tending to $0$ such that the map
$$f_n = \sqrt{\frac {\psi_n} { \| \psi_n \|_1}}$$
 is $(K_n, \vareps'_n)$-almost invariant for the $L_2$-norm. Notice that  $\| f_n \|_2 = 1$ by construction. 
  
Finally, let  $\tilde f_n = f_n - \int_X f_n dm$ denote the projection of $f_n$ onto $L_2^0(X)$. Since the support of $f_n$ lies entirely in the complement of $X_n$, we have
 $$\begin{array}{rcl}
 \int_X f_n   d m= \int_{X - X_n} f_n dm  =   \langle 1_{X - X_n}, f_n\rangle \le \| 1_{X - X_n}\|_2\cdot\| f_n\|_2=m(X - X_n)^{\frac{1}{2}}
 \end{array}$$
 by the Cauchy--Schwarz inequality. Therefore $\| \tilde f_n - f_n \|_2 \to 0$ as $n$ tends to infinity. This implies that $(\tilde f_n/\| \tilde f_n\|_2)$  is  a sequence of  almost invariant vectors of norm one in $L_2^0(X)$, as desired.
\end{proof}


\section{Exotic examples}

The goal of this section is to prove Theorem~\ref{theo:example}. This will be obtained by specializing a general construction, described in the first subsection below.

\subsection{Lattices in restricted products}\label{sec:RestricedProduct}

Let $(G_n)_{n \geq 0}$ be a sequence of locally compact groups. For each $n \geq 0$, let $U_n < G_n$ be a compact open subgroup. To the sequence $(G_n, U_n)_{n \geq 0}$, one associates the \textbf{restricted product}, defined as follows:
$$\prod_n{}' (G_n, U_n) = \{(g_n) \in \prod_{n \geq 0} G_n \; | \; g_n \in U_n \text{ for all but finitely many  $n$'s}\}.$$
The group $G = \prod'_n (G_n, U_n) $ commensurates its subgroup $U =  \prod_n U_n $, which is compact. Therefore $G$  carries a unique locally compact group topology that makes $U$ a compact open subgroup.

We now describe how one can construct lattices in the restricted product $G$ from lattices in the various factors $G_n$. For each $n \geq 0$, let $\Gamma_n$ be a subgroup of $G_n$ intersecting  $U_n$ trivially. In particular $\Gamma_n$ is discrete in $G_n$. Assume that $\Gamma_n$ is of finite covolume, and let $c_n$ denote that covolume with respect to the normalization of the Haar measure on $G_n$ that gives measures $1$ to the compact open subgroup $U_n$. Then we have the following.

\begin{prop}\label{prop:ConstrScheme}
Retain the notation as above.
\begin{enumerate}[(i)]
\item $\Gamma = \bigoplus_n \Gamma_n$ is a discrete subgroup of $G$.

\item $\Gamma$ is a lattice in $G$ if and only if the sequence $C_k = \prod_{n=0}^k c_n$ converges to a  finite limit as $k$ tends to infinity.

\item $\Gamma$ is a uniform lattice in $G$ if and only if $\Gamma_n$  is uniform in $G_n$ for all $n \geq 0$ and  there is some $N$ such that $c_n=1$ for all $n \geq N$.
\end{enumerate}
\end{prop}
\begin{proof}
(i) By construction $U = \prod_n U_n$ is a compact open subgroup of $G$ that intersects $\Gamma$ trivially. Thus $\Gamma$ is discrete in $G$.

\medskip \noindent
(ii) For each $k \geq 0$, let $O_k < G$ be the subgroup of $G$ consisting of those sequences $(g_n)_{n \geq 0}$ such that $g_n \in U_n$ for all $n > k$. Thus  $(O_k)$ forms an ascending chain of open subgroups of $G$, whose union is the whole of $G$. For each $n \geq 0$, let $F_n$ be a fundamental domain for $\Gamma_n$ in $G_n$. Since $\Gamma_n \cap U_n = \{1\}$ we may assume that $F_n$ contains $U_n$ for all $n$. Then for each $k \geq 0$, the set
$$\Omega_k = F_0 \times \dots \times F_k \times U_{k+1} \times U_{k+2} \times \dots $$
is a fundamental domain for $\Lambda_k =  \Gamma \cap O_k$ in $O_k$.
Normalize the Haar measure on $G$ so that $U$ has measure~$1$. Then we have
$$\vol(O_k/  \Lambda_k) = \vol(\Omega_k) = \prod_{n=0}^k \lambda_n(F_n) = \prod_{n=0}^k c_n = C_k$$
for all $k$, where $\lambda_n$ denotes the Haar measure on $G_n$. Therefore we deduce that
$$\vol(G/\Gamma) = \lim_{k \to \infty} \vol(O_k/  \Lambda_k)=  \lim_{k \to \infty} C_k.$$
The desired assertion follows.

\medskip \noindent
(iii) The necessity of the condition that $\Gamma_n$ is uniform in $G_n$ for all $n$ is clear; we assume henceforth that it holds. In particular we may assume that $F_n$ is compact for all $n$. Therefore the sets $\Omega_k$ are all compact and form an ascending chain whose union $\Omega = \bigcup_{k \geq 0} \Omega_k$ is a fundamental domain for $\Gamma$ in $G$.

If $c_n=1$ for all $n \geq N$, then $F_n = U_n$ for all $n \geq N$ (up to a set of measure zero) so that $\Omega = \Omega_k$ for all $k \geq N$.  Thus $\Gamma$ is uniform in $G$.

Conversely, if $\Gamma$ is uniform,   then it has a compact fundamental domain, which must therefore be contained in $O_N$ for some sufficiently large $N$.     Then for all $n \geq N$ we have
$$\lim_{k \to \infty} C_k = \vol(G/\Gamma) = \vol(\Omega_n/\Lambda_n) = C_n,$$
so that $c_n = 1$ for all $n \geq N$.
\end{proof}

\subsection{Proof of Theorem \ref{theo:example}}

Retain the notations of Theorem \ref{theo:example}.
For each $n \geq 0$, we set $T_n =  \FF_{q_n}$, $U_n = \FF_{q_n}^*$ and $G_n = T_n \rtimes U_n$. We normalize the counting measure on $G_n$ so that $U_n$ has measure $1$; thus each element of $G_n$ has measure $\frac 1 {q_n-1}$. Notice that $G$ is nothing but the restricted product of the sequence $(G_n, U_n)$.

For each $n \geq 0$, we further define $S_n \cong  \FF_{q_n}^*$ as the image of the injective homomorphism
$$ \FF_{q_n}^* \to G_n =  \FF_{q_n} \rtimes  \FF_{q_n}^* : x \mapsto (x-1, x).$$
Identifying  $G_n $ with its canonical image in $G$,   we view $T_n$ and $S_n$ as finite subgroups of $G$.

\medskip
Let now $A \subset \NN$ be a set of non-negative integers. For each $n \geq 0$ we set $\Gamma_n = S_n$ if $n \in A$ and $\Gamma_n = T_n$ otherwise. Finally, we define $\Gamma_A = \bigoplus_{n \geq 0} \Gamma_n$. Then Proposition~\ref{prop:ConstrScheme} ensures that $\Gamma_A$ is a lattice in $G$ if and only if the sequence
$(\prod_{n \in A, n \leq k} \frac{q_n}{ q_n-1})_{k \geq 0}$ converges to a finite value. This happens if the series $\sum_n \log ( \frac{q_n}{ q_n-1}) = \sum_n \log( 1 + \frac 1 {q_n-1})$ converges. Using Taylor expansion of the function $x \mapsto \log(1+x)$ (which has convergence radius equal to~$1$), we have
$$\begin{array}{rcl}
 \log( 1 + \frac 1 {q_n-1}) & = & \sum_{m>0} \frac{(-1)^{m+1}} m (\frac 1 {q_n-1})^m\\
& \leq &  \sum_{m>0} (\frac 1 {q_n-1})^m\\
& = &  \frac{ 1/ (q_n-1)}{1-   1/( {q_n-1})}\\
& = &  \frac 1 {q_n - 2}\\
& < & \frac 2 {q_n }
\end{array}$$
provided $q_n > 4$.
Therefore the hypothesis $ \sum_n \frac 1 {q_n } < \infty$ ensures that  the series $\sum_{n } \log( 1 + \frac 1 {q_n-1})$ converges, and hence  that $\Gamma_A$ is a lattice.

  Proposition~\ref{prop:ConstrScheme} also shows that    $\Gamma_A$ is uniform if and only if the sequence $(\prod_{n \in A, n \leq k} \frac{q_n}{ q_n-1})_{k \geq 0}$ is eventually constant.  This in turn is equivalent to the fact that the set $A$ is finite, so $\Gamma_A$ is non-uniform for any infinite set $A$. 
  
Let now $B \subset   \NN$ be another infinite set, and assume that $\Gamma_A$ and $\Gamma_B$ are commensurable, i.e. there exists $g \in G$ such that $g \Gamma_A g^{-1} \cap \Gamma_B$ is of finite index in $g \Gamma_A g^{-1} $ and in $ \Gamma_B$. Since $\bigoplus_{n \not \in A} T_n$ is a normal subgroup of $G$ contained in $\Gamma_A$, we see that $\Gamma_B$ contains a finite index subgroup of $\bigoplus_{n \not \in A} T_n$. Similarly $\Gamma_A$ contains a finite index subgroup of $\bigoplus_{n \not \in B} T_n$. It follows that the symmetric difference $A \triangle B$ is finite. 
Thus, by letting $A$ vary over infinite subsets of $\NN$ with infinite complements, we obtain an uncountable family of pairwise non-commensurable  non-uniform lattices, as desired.

\medskip
Let now $\Gamma < G$ be a uniform lattice. We want to show that $\Gamma$ is commensurable with $\Lambda$. By the same argument as before, the cocompactness of $\Gamma$ implies that the sequence $ \vol(O_n/O_n \cap \Gamma)$ is eventually constant. For each $n \geq 0$, let $p_n \colon G \to G_n$ be the natural projection, and for $N \geq 0$ let $p_{> N} $ denote the natural projection of $G$ onto the restricted product $\prod'_{n > N} (G_n, U_n)$, defined as in Section~\ref{sec:RestricedProduct}. Let also $\Gamma_n = O_n \cap \Gamma$. Since $ \vol(O_n/O_n \cap \Gamma)$ is eventually constant, there exists $N$ such that $q_{n} = |O_{n} : O_{n-1}|  = |\Gamma_{n} : \Gamma_{n-1}|$ for all $n > N$. Via the projection $p_n \colon G \to G_n$, the group $G$ acts on the affine line $\FF_{q_n}$. Let us consider the induced action of $\Gamma_n$. By the Orbit Counting Lemma, we have $|\Gamma_n| = |\Gamma_{n-1}|. |\Gamma_n(0)|$. For $n>N$, we have $|\Gamma_{n} : \Gamma_{n-1}|=q_n$, so that $\Gamma_n$ is transitive on $\FF_{q_n}$. Since $G_n \cong\FF_{q_n} \rtimes \FF_{q_n}^*$ has a unique subgroup of order~$q_n$, namely $T_n$, we infer that $T_n \leq p_n(\Gamma_n)$ for all $n >N$. Similarly, considering the product action of $G$ on the finite product $\prod_{N<n<M} \FF_{q_n}$ of affine lines for $M>N$, we deduce that $\bigoplus_{N<n<M} T_n \leq (\bigoplus_{N<n<M}  p_n)(\Gamma_M)$ for all $M >N$. In particular $p_{>N}(\Lambda) = \bigoplus_{n>N} T_n \leq  p_{>N}(\Gamma)$. Since $p_{>N}$ is a continuous homomorphism with finite kernel, it maps a lattice to a lattice. It follows that $p_{>N}(\Lambda) $ is of finite index in $p_{>N}(\Gamma)$. Considering now finite index subgroups $\Lambda' \leq \Lambda$ and $\Gamma' \leq \Gamma$ such that the restriction of $p_{>N}$ to $\Lambda' $ and $\Gamma'$ is injective, we deduce   that $p_{>N}(\Gamma')$ and $p_{>N}(\Lambda')$ are commensurable, and hence that $\Gamma$ and $\Lambda$ are commensurable. \qed


\begin{rem*}
Recall the following compactness criterion \cite[Theorem 1.12]{Raghu}: A lattice $\Gamma$ in a locally compact group $G$ is non-uniform if and only if there is a sequence of elements $\gamma_n\in \Gamma\setminus\{1\}$ such that for some appropriate $g_n\in G$ we have $g_n\gamma_ng_n^{-1}\to 1$. One may view such a sequence $(\gamma_n)$ as an \textbf{approximated unipotent} sequence.
In the classical case where $G$ is a semisimple Lie group it follows from the celebrated Kazhdan--Margulis theorem \cite{Kazhdan--Margulis68} that every non-uniform lattice $\Gamma\le G$ admits a non-trivial unipotent, i.e. an element $\gamma\in \Gamma\setminus\{1\}$ whose conjugacy class contains the identity in its closure. In the general case, where $G$ is a locally compact group, one may call a non-trivial element $\gamma\in G$ \textbf{pseudo-unipotent} if the identity is contained in the closure of its conjugacy class, i.e. if there are $g_n\in G$ for which $g_n\gamma g_n^{-1}\to 1$. Then one may ask whether the analogue of the Kazhdan--Margulis theorem holds for $G$, i.e. whether every non-uniform lattice admits a pseudo-unipotent.

The group $G$ constructed above (in the proof of Theorem \ref{theo:example}) is an example showing that this is not always true. Indeed for every non-trivial $g \in G$, there exists $n$ such that the image of $g$ under the continuous quotient map $G \to G_n$ is non-trivial. So the conjugacy class of $g$ in $G$ cannot accumulate to the identity hence $G$ admits no pseudo-unipotents (more generally: a residually discrete locally compact does not contain pseudo-unipotent elements). Therefore a non-uniform lattice $\Gamma\le G$ provides an example of a lattice with approximated unipotents but no pseudo-unipotent.
\end{rem*}


\section{Auxiliary assertions on locally compact groups}\label{sec:aux}

 In the following we abbreviate the expression \emph{totally disconnected locally compact} by \textbf{\tdlc}

\subsection{Compact generation of cocompact subgroups}

The following well known fact is important.

\begin{prop}\label{prop:CompactGen}
A cocompact closed subgroup of a compactly generated locally compact group is itself compactly generated.
\end{prop}
\begin{proof}
See \cite{Macbeath-Swierczkowski59}.
\end{proof}

\subsection{Compactly presented groups}

A locally compact group $G$ is called \textbf{compactly presented} if there is a surjective homomorphism $\theta \colon F_X \rightarrow G$, where $F_X$ is the abstract free group on a set $X$, so that $\theta(X)$ is compact in $G$ and the kernel of $\theta$ is generated by words in the alphabet $X \cup X\inv$ of uniformly bounded length.

If $G$ is discrete, this is equivalent to the condition that $G$ be finitely presented. The following result, well-known in the discrete case, is due to Abels in the locally compact setting.

\begin{prop}\label{prop:CompactlyPresented:DiscreteExtension}
Let $\varphi \colon \tilde G \to G$ be a continuous homomorphism of locally compact groups. If $G$ is compactly presented and $\tilde G$ is compactly generated, then $\ker(\varphi)$ is compactly generated as a normal subgroup of $\tilde G$.
\end{prop}

\begin{proof}
See Theorem 2.1 in \cite{Abels}.
\end{proof}

The following result, also due to Abels, will be used later.

\begin{prop}\label{prop:Nilpotent->CptlyPres}
Every compactly generated nilpotent locally compact group is compactly presented.
\end{prop}
\begin{proof}
See Theorem B in \cite{Abels}.
\end{proof}

\subsection{Discrete cocompact normal subgroups}

\begin{lem}\label{lem:DiscreteNormal}
Let $G$ be a compactly generated locally compact group and $N$ be a discrete cocompact normal subgroup. Then the centralizer of $N$ in $G$, namely $C_G(N)$, is an open normal subgroup, and the product $N C_G(N)$ is open and of finite index in $G$.

If in addition $G$ is totally disconnected, then $G$ has a compact open normal subgroup intersecting $N$ trivially.
\end{lem}
\begin{proof}
By Proposition~\ref{prop:CompactGen}, the group $N$ is finitely generated. Since $N$ is discrete, every $g \in N$ has a discrete conjugacy class, hence an open normalizer in $G$. Since $N$ is  generated by a finite set, say $S$, the centralizer $C_G(N) = \bigcap_{s \in S} C_G(s)$ is open. It is normal in $G$ since $N$ is so. Since $G/N$ is compact, the image of the open subgroup $C_G(N)$ in $G/N$ is of finite index. Hence $G/N C_G(N)$ is finite.

We now assume that $G$ is totally disconnected. Then $C_G(N)$ contains a compact open subgroup $U$ intersecting $N$ trivially. The group $NU$ is open and cocompact in $G$, hence of finite index. Since $N$ commutes with $U$, we infer that the conjugacy class of $U$ in $G$ is finite. Hence $\bigcap_{g \in G} g U g^{-1}$ is a compact open normal subgroup of $G$ intersecting $N$ trivially.
\end{proof}

\subsection{Compactly generated nilpotent tdlc groups}

\begin{prop}\label{prop:nil-Willis}
Let $G$ be a compactly generated nilpotent  \tdlc group. Then $G$ has a basis of identity neighborhoods consisting of compact open normal subgroups.
\end{prop}
\begin{proof}
See \cite{Willis_nil}.
\end{proof}

Notice that Proposition~\ref{prop:Nilpotent->CptlyPres} can be deduced from Proposition~\ref{prop:nil-Willis} in the totally disconnected case, by invoking the classical Mal'cev theory, which implies that finitely generated nilpotent groups are finitely presented.

\subsection{A cocompactness criterion}

The following simple lemma will be used repeatedly in later sections.

\begin{lem}\label{lem:HUGA}
Let $G$ be a locally compact group. Let $H,A,U\le G$ be closed subgroups such that
\begin{itemize}
 \item $A$ is normal in $G$,
 \item $U$ is open in $G$ and
 \item $HA$ is dense in $G$.
\end{itemize}
Then $G/H$ is compact if and only if $UA/(UA\cap H)$ is compact.
\end{lem}

\begin{proof}
Since $HA$ is dense and $U$ is open, $HUA=HAU=G$, hence a fundamental domain of $UA\cap H$ in $UA$ is also a fundamental domain for $H$ in $G$.
\end{proof}

\subsection{Locally elliptic groups}\label{sec:LE}

The notions and results of this section, due to Platonov \cite{Platonov}, will be used frequently without notice.

Given a locally compact group $G$, a subgroup $H \leq G$ is called \textbf{locally elliptic} if every finite subset of
$H$ is contained in a compact subgroup of $G$. As shown in \cite{Platonov} this is equivalent to the requirement that every compact subset of $H$ is contained in a compact subgroup of $G$. Therefore, if $G$ is locally elliptic and second countable, then $G$ is a union of a countable ascending chain of compact open subgroups.

By \cite{Platonov}, the closure of a locally elliptic subgroup is locally elliptic, and an extension of a locally elliptic group by a locally elliptic group is itself locally elliptic. In particular any locally compact group $G$ has a largest closed normal subgroup which is locally elliptic, called the \textbf{LE-radical} of $G$ and denoted by $\LF(G)$. It is a closed characteristic subgroup of $G$ and the quotient $G/\LF(G)$ has trivial LE-radical.

\subsection{Locally $p$-elliptic groups}

The results in this section are variations on Platonov's results on locally elliptic groups which we shall need later.

Let $p$ be a prime. Given a \tdlc group $G$, a subgroup $H \leq G$ is called \textbf{locally $p$-elliptic} if every finite subset of
$H$ is contained in a pro-$p$ (hence compact) subgroup of $G$. A locally $p$-elliptic subgroup is thus clearly locally elliptic. It turns out that the properties of locally elliptic groups mentioned above have a direct analogue for the class of locally $p$-elliptic groups. We record them in the following.

\begin{prop}\label{prop:Locally-p-elliptic}
Let $G$ be a \tdlc group.

\begin{enumerate}[(i)]
\item  In a locally $p$-elliptic subgroup $H \leq G$,  every compact subset is contained in a pro-$p$ subgroup of $G$.

\item If $H \leq G$ is   locally $p$-elliptic, then so is the closure $\overline H$.

\item If $H \leq G$ is a closed   normal locally $p$-elliptic subgroup, then $G$ is locally $p$-elliptic if and only if $G/H$ is so.

\item Every \tdlc group $G$ has a largest locally $p$-elliptic normal subgroup, denoted by $\LFp(G)$, called the \textbf{$p$-radical}. It is a closed characteristic subgroup of $G$, and the quotient $G/\LFp(G)$ has trivial $p$-radical.
\end{enumerate}

\end{prop}

\begin{proof}
(i). Let $H \leq G$ be a locally $p$-elliptic subgroup and  $C \subset   H$ be a compact subset. We must prove that $V = \overline{\langle C \rangle}$ is a pro-$p$ subgroup of $G$. We know from \S\ref{sec:LE} that $V$ is a compact, hence profinite, subgroup. If $V$ is not pro-$p$, then $V$ maps continuously onto a finite group whose order is not a power of $p$. In particular $V$ has an open subgroup $W$ admitting a continuous surjective map $\varphi \colon W \to \ZZ/q$  onto a cyclic group of order $q$, for some prime $q \neq p$. Since $\langle C \rangle$ is dense in $V$ and $W$ is open, it follows that $\langle C \rangle \cap W$ is dense in $W$. Therefore there exists $w \in \langle C \rangle \cap W$ such that $\varphi(w)$ generates $\ZZ/q$. Let then $c_1, \dots, c_m \in C$ such that $w \in \langle c_1, \dots, c_m \rangle$. Since $H$ is locally $p$-elliptic, the group $\overline{\langle c_1, \dots, c_m\rangle}$ is pro-$p$. Thus the group $ \overline{\langle c_1, \dots, c_m\rangle} \cap W$, which contains $w$, is also pro-$p$. It follows that the cyclic group $\varphi(W)$ is a quotient of the pro-$p$ group $ \overline{\langle c_1, \dots, c_m\rangle} \cap W$, and is thus  a $p$-group. This contradiction finishes the proof of (i).

\medskip \noindent
(ii) We know from  \S\ref{sec:LE} that $\overline H$ is locally elliptic. If it is not locally $p$-elliptic, then it contains a compact open subgroup which is not pro-$p$. This implies again that it contains a compact open subgroup $W$ which maps continuously onto a cyclic group of prime order $q \neq p$. Since $H \cap W$ is dense in $W$, we may conclude as in the proof of (i).

\medskip \noindent
(iii) Every continuous quotient of a locally $p$-elliptic group is locally $p$-elliptic. Thus the   `only if' part is clear. Suppose conversely that $G/H$ is locally $p$-elliptic. Let $U \leq G$ be a compact open subgroup. Then $U$ decomposes as an extension of $UH/H$ by $U \cap H$. Both of these profinite groups are pro-$p$ by hypothesis, in view of (i). Therefore $U$ is pro-$p$, and $G$ is indeed locally $p$-elliptic.

\medskip \noindent
(iv) The class of locally $p$-elliptic subgroups of $G$ is stable under taking directed unions. In particular it contains maximal elements by Zorn's lemma. Such a maximal element is closed by (ii). The other required assertions now follow from (iii).
\end{proof}

\subsection{\{Compact-by-discrete\}-by-compact groups}

Given a (locally compact) group $G$ and a (closed) normal subgroup $N \leq G$ such that $N$ satisfies an algebraic/topological property $A$ and $G/N$ satisfies a property $B$, then we say that $G$ is a \textbf{A-by-B group}.

\begin{lem}\label{lem:compact-by-discrete-by-compact}
Let $G$ be a compactly generated locally compact group with a cocompact closed normal subgroup $N$. Assume that $N$ has a compact relatively open normal subgroup $V$. Then $V$ is contained in a compact normal subgroup of $G$.

If in addition $G$ is totally disconnected, then $V$ is contained in a compact open normal subgroup of $G$.
\end{lem}

\begin{proof}
Consider the  LE-radical $\LF(N)$, which is normal in $G$.
If $K$ denotes a compact subset of $G$ such that $G = KN$, then every conjugate of $V$ in $G$ is contained in the compact subset $KVK \cap \LF(N)$. It follows that the normal closure of $V$ in $G$ is contained in a compactly generated subgroup of $\LF(N)$, which must thus be compact. Thus $V$ is contained in a compact normal subgroup $W$ of $G$.

If $G$ is totally disconnected, Lemma~\ref{lem:DiscreteNormal} ensures that $G/W$ has a compact open normal subgroup. Its preimage in $G$ is then a compact open normal subgroup containing $V$.
\end{proof}

\subsection{Serre's covolume formula}

\begin{lem} \label{lem:serre}
Let $G$ be a locally compact group acting on the measured space $(X,\mu)$ preserving $\mu$.
Fix a point $x\in X$ and denote $H=\mathrm{Stab}_G(x)$.
For $K_1,K_2$ compact open subgroups of $G$ we have
\begin{equation} \label{eq:comme}
\vol_G(K_2)\cdot\vol_{H}(K_1\cap H)\cdot \mu(K_1x)=\vol_G(K_1)\cdot \vol_H(K_2\cap H)\cdot\mu(K_2x),
\end{equation}
where $\vol_G$ is a left Haar measure on $G$ and $\vol_H$ is a left Haar measure on $H$
(note that Equation~(\ref{eq:comme}) does not depend on the choices of these measures).
\end{lem}

\begin{proof}
Observe that among compact open subgroups of $G$, the collection of pairs satisfying Equation~\ref{eq:comme} forms an equivalence relation. Notice moreover that for any two compact open subgroups $K_1, K_2$ in $G$, the group $K_1$ (resp. $K_2$) has an open normal subgroup $K'_1$ (resp. $K'_2$) contained in  $K_1 \cap K_2$. Therefore, in order to prove that  Equation~(\ref{eq:comme}) holds for all pairs $(K_1, K_2)$ of compact open subgroups of $G$, it suffices to prove (\ref{eq:comme}) for the  subcollection 
$\{(K_1,K_2)\mid K_2\lhd K_1\}$. We assume henceforth that   $K_2$ is normal (and of finite index) in $K_1$.

We get that $K_1x\to K_1/K_2\mathrm{Stab}_{K_1}(x)$ is a $K_1$- equivariant map with fibers isomorphic to $K_2x$. As $\mathrm{Stab}_{K_1}(x)=K_1\cap H$, we get the equation
\[ \mu(K_1x)=[K_1:K_2(K_1\cap H)]\cdot\mu(K_2x).\]
Multiplying corresponding sides with the equation  
\[ [K_1\cap H:K_2 \cap H]
=[K_1\cap H:K_2\cap (K_1\cap H)]  
= [K_2(K_1\cap H):K_2]
\]
and substituting 
\[ [K_1:K_2]=[K_1:K_2(K_1\cap H)]\cdot [K_2(K_1\cap 
H):K_2], \]
we obtain
\[ [K_1\cap H:K_2 \cap H]\cdot \mu(K_1x) = [K_1:K_2]\cdot \mu(K_2x) \]
which is equivalent to Equation~(\ref{eq:comme}).
\end{proof}

The following result is  well known in the special case where $H$ is discrete; it is sometimes called \emph{Serre's covolume formula} in that case. 
A published reference where Serre's covolume formula for discrete groups is explicitly stated and proved is in M.~Bourdon's paper \cite[Proposition~1.4.2(b)]{Bourdon}.

\begin{prop}\label{prop:vol}
Let $G$ be a unimodular  \tdlc   group and fix a compact open subgroup $K<G$.
Let $H \leq G$ be a closed unimodular subgroup.
We normalize the Haar measure $\vol_G$ on $G$ and $\vol_H$ on $H$ so that $\vol_G(K) = \vol_H(K\cap H)=1$.
We let $x\in G/H$ be the base point and 
we take the compatible normalization of the Haar measure $\mu$ on $G/H$ so that $\mu(Kx)=1$. 
Then for any measurable section $s:G/H\to G$, identifying $G$ with $s(G/H)H$, we have $\vol_G=s_*\mu\times \vol_H$. Moreover the total measure of $G/H$ is given by the formula
$$
 \mu(G/H)=\sum_{t\in \Omega}\frac{1}{\vol_H (H_t)},
$$
where $\Omega$ denotes a fundamental domain for the right $H$ action on (the discrete space) $K\backslash G$
and for $t\in \Omega$, $H_t$ denotes the stabilizer of $t$ in $H$ (which is compact open).
\end{prop}

\begin{proof}
For the first statement see Lemma~\ref{lem:haarnorm}.
For the second statement, identifying $\Omega=K\backslash G/H$, we need to show that
for $g\in G$, the $\mu$-measure of the $K$ orbit of $gx\in G/H$ is given by reciprocal of the $\vol_H$-measure of the stabilizer in $H$ of $Kg\in K\backslash G$, which is $g^{-1}Kg\cap H$.
Fix $g\in G$.
By the $G$-invariance of $\mu$ and by applying Lemma~\ref{lem:serre} to the commensurable subgroups $g^{-1}Kg,K<G$, we have
\[ \mu(Kgx)=\mu(g^{-1}Kgx)= \]
\[ \vol_G(g^{-1}Kg)\vol_G(K)^{-1}\vol_H(K\cap H)\vol_H(g^{-1}Kg\cap H)^{-1}\mu(Kx). \]
By our normalization, $\mu(Kx)=1/\vol_H(H\cap K)$ and by unimodularity, $\vol_G(K)=\vol_G(g^{-1}Kg)$.
It follows that indeed
\[ \mu(Kgx)=\vol_H(g^{-1}Kg\cap H)^{-1}. \qedhere \]
\end{proof}

This implies the well known fact that a torsion-free lattice in a \tdlc group must be cocompact. More generally, we have the following immediate consequence of Proposition \ref{prop:vol}.

\begin{cor}\label{cor:bddvol}
Let $H$ be a  totally disconnected, locally compact, unimodular group. Assume that there is a uniform upper bound on the volumes of compact open subgroups of $H$. Then, given any t.d.l.c group $G$   containing $H$ as a closed cofinite subgroup,  the group $H$ is cocompact in $G$.
\end{cor}

An interesting class of examples for groups $H$ satisfying the assumption above is the class of semi-simple groups over non-Archimedean local fields, which have finitely many conjugation classes of maximal compact open subgroups.

\section{Amenable groups with property (M)}\label{sec:PositiveResults}

The goal of this section is to establish  property (M) for various special (and natural) classes of amenable locally compact groups.

\subsection{Compact extensions and property (M)}

The following easy fact shows that property (M) is insensitive to dividing out compact normal subgroups.

\begin{lem}\label{lem:CompactKernel}
Let $G$ be a locally compact group and $K$ be a compact normal subgroup of $G$. Then $G$ has (M) if and only if $G/K$ has (M).
\end{lem}

\begin{proof}
The `only if' is clear and does not require the compactness of $K$. For the converse, the point is that the compactness of $K$ ensures that projection $G \to G/K$ is closed.
\end{proof}

Note however, that the examples provided by Theorem~\ref{theo:example} show that property (M) is affected by taking an extension by a compact group. Indeed, discrete groups and abelian locally compact groups all satisfy (M), but Theorem~\ref{theo:example} shows that a locally compact group with a cocompact normal subgroup that is abelian and discrete may fail to satisfy (M).

\subsection{Central extensions}

\begin{prop}\label{prop:CentralExt}
Property (M) is inherited by central extensions.

In other words, if $\tilde G$ is a locally compact group and $Z$ a closed central subgroup such that $\tilde G/Z = G$ has property (M), then $\tilde G$ has property (M).
\end{prop}

\begin{proof}
Let $H\le\ti G$ be a cofinite subgroup and denote $\ti H = \overline{ZH}$. Then $\ti H$ is a closed subgroup of finite covolume in $\ti G$ and hence $\ti H/Z$ is cofinite in $\ti G/Z\cong G$. Since $G$ has property $(M)$ we deduce that $(\ti G/Z)/ (\ti H/Z)\cong \ti G/\ti H$ is compact.
Thus, it suffices to show that $\ti H/H$ is compact. But $H$ is normal and cofinite in $\ti H$ hence the group $\ti H/H$ carries a finite $\ti H/H$ invariant measure. Having a finite Haar measure, the latter group is indeed compact. 
\end{proof}

\subsection{Direct products}

\begin{prop}\label{prop:DirectProd}
	Let $G = G_1 \times G_2$ be the direct product of two locally compact groups with Property (M). If $G_1$ is totally disconnected, then $G$ has property (M).
\end{prop}
\begin{proof}
Let $H\le G$ be a cofinite subgroup. Since $G_1$ has (M), the closure of the projection of $H$ to $G_1$ is cocompact. We may thus replace $G_1$ by that closure and assume that  the projection of $H$ to $G_1$ is dense. Let then $U_1$   be a compact open subgroup of $G_1$ and set $U = U_1 \times G_2$. Then the projection $U \to G_2$ is proper, so that $H\cap U$ is cocompact in $U$ because $G_2$ has property (M). It now follows from Lemma~\ref{lem:HUGA} that $H$ is cocompact in $G$.
\end{proof}

\subsection{Discrete groups}

\begin{prop}\label{prop:discrete}
\begin{enumerate}[(i)]
\item Discrete groups have property (M).

\item Compactly generated discrete-by-compact groups have  property (M).

\item A locally compact group with an open normal subgroup having (M) also has (M).
\end{enumerate}

\end{prop}

\begin{proof}
The assertion (i) follows from the fact that a discrete homogeneous space has finite measure only if it is finite.

For assertion (ii), let $G$ be a locally compact group and $N$ be a  discrete normal cocompact subgroup of $G$. Let also $H \leq G$ be a cofinite subgroup.
Let $U = C_G(N)$, so that $U$ is open by Lemma~\ref{lem:DiscreteNormal}. Then $HU$ is an open cofinite subgroup of $G$, hence of finite index. We may thus assume without loss of generality that $G = HU$. By Lemma~\ref{lem:HUGA}, it suffices to show that $U \cap H$ is cocompact in $U$. Now we observe that $U \cap N$ is contained in the center of $U$. Moreover $U/U \cap N \cong UN/N$ since $U$ is open, so that $U/U \cap N$ is compact. This shows that $U$ is a central extension of a compact group. Therefore $U$ has property (M) by Proposition~\ref{prop:CentralExt}, so $U \cap H$ is cocompact in $U$, as required.

Assertion (iii) follows from (i) together with Lemma~\ref{lem:HUGA}.
\end{proof}

\begin{cor}\label{cor:compact-by-discrete}
Compactly generated \{compact-by-discrete\}-by-compact groups have  property (M).
\end{cor}

\begin{proof}
Let $G$ be a locally compact group with a compactly generated  closed cocompact normal subgroup $N$. Assume that $N$ has a compact relatively open normal subgroup $V$. We must show that $G$ has property (M).

By Lemma~\ref{lem:compact-by-discrete-by-compact}, the group $V$ is contained in a compact normal subgroup $W$ of $G$.
By Lemma~\ref{lem:CompactKernel}, we may replace $G$ by $G/W$ and assume therefore that $V$ is trivial. Hence $N$ is discrete, and the desired assertion follows from Proposition~\ref{prop:discrete}(ii).
\end{proof}

\subsection{Nilpotent groups}

\begin{prop}\label{prop:nilpotent}
\begin{enumerate}[(i)]
\item Nilpotent locally compact groups have property (M).

\item Compactly generated nilpotent-by-compact locally compact groups have property (M).
\end{enumerate}

\end{prop}

\begin{proof}
(i)  follows by induction on the nilpotency degree,  using Proposition~\ref{prop:CentralExt}.

For (ii), we need to show that  compactly generated nilpotent-by-compact locally compact groups  have property (M). We first observe that nilpotent-by-compact groups are amenable. Moreover, the class of nilpotent-by-compact groups is stable under taking quotient groups: indeed, given $G$ with a closed cocompact nilpotent subgroup $N$ and $H$ a closed normal subgroup of $G$, we see that $N \leq \overline{NH} \leq G$ so that $\overline{NH}$ is cocompact in $G$. Since the image of $N$ in $G/H$ is nilpotent, so is its closure. So $\overline{NH}/H$ is a closed cocompact nilpotent normal subgroup of $G/H$.

In view of those observations, it suffices by Theorem~\ref{theo:reduction} to show that compactly generated nilpotent-by-compact \tdlc groups  have property (M). By \cite{Willis_nil}, a compactly generated nilpotent \tdlc group is compact-by-discrete, so that the desired assertion now follows from Corollary~\ref{cor:compact-by-discrete}.
%
\end{proof}

\subsection{Solvable Noetherian groups}

A locally compact group is called \textbf{Noetherian} if it satisfies an ascending chain condition on open subgroups. A condition which is equivalent to Noetherianity is that  every open subgroup is compactly generated. We warn the reader that our choice of terminology disagrees with the terminology introduced by Y. Guivarc'h, who chose to call a locally compact group \emph{Noetherian} if each of its \emph{closed} subgroup is compactly generated (see \cite{Guivarch}*{\S{}III, p.~346}) . The latter condition is of course considerably stronger than the   Noetherian condition as defined above. An argument in favour of our own choice is that a locally compact group which is Noetherian in Guivarc'h's sense can however fail to satisfy the ascending chain condition on closed subgroups: the simplest examples  are provided by $G = \RR$ or $G = S^1$. However, we shall observe in Proposition~\ref{prop:Polycyclic} below that a solvable locally compact group which is Noetherian in the sense above is also Noetherian in the sense of Guivarc'h. 

Henceforth, we exclusively refer to the term \emph{Noetherian} in the sense introduced above, namely as the ACC on open subgroups. Clearly, every connected group is Noetherian, since it has no proper open subgroup. Compact groups are also Noetherian since open subgroups have finite index. A theorem of Tits, proved by Prasad, asserts that in a simple locally compact group over a local field, every proper open subgroup is compact (this can be deduced from the Howe--Moore property); clearly this implies Noetherianity. Notice moreover that Noetherianity is inherited by quotients (since the projection map is open). However it is not inherited by closed subgroups in general.  For a general result on the structure of (non-discrete) locally compact Noetherian groups, we refer to \cite[Theorem~C]{CaMo}.

A discrete solvable group is Noetherian if and only if it is polycyclic (see \cite[Ch. 3 \& 4]{Raghu} for generalities about polycyclic groups).

The following result generalizes Mostow's theorem.

\begin{thm}\label{thm:Noetherian}
\{Solvable Noetherian\}-by-compact locally compact groups have  property $(M)$.
\end{thm}

It can be seen that every compactly generated nilpotent locally compact group is Noetherian (this follows by combining the Noetherianity of solvable Lie groups with the main result of \cite{Willis_nil}). However, not every compactly generated solvable locally  compact group is Noetherian: this is already seen within discrete groups, since not every finitely generated metabelian group is polycyclic, as highlighted by the lamplighter.

\medskip
The proof of Theorem~\ref{thm:Noetherian} relies on the following observation.

\begin{prop}\label{prop:SolubleNoether}
Let $G$ be a solvable \tdlc group. If $G$ is Noetherian, then it admits a compact open normal subgroup.
\end{prop}

\begin{proof}
We work by induction on the solubility degree of $G$, the case of abelian groups being trivial.

Let thus $A < G$ be the last non-trivial term of the derived series. Since $G/A$ is solvable, Noetherian and totally disconnected, the induction hypothesis implies that it has a compact open normal subgroup $K$. Let $B < G$ be the preimage of $K$ in $G$. Thus $B$ is an open normal subgroup of $G$ containing $A$ as a closed cocompact normal subgroup. Note that $B$ is compactly generated since $G$ is Noetherian. Thus $A$ is compactly generated as well by Proposition~\ref{prop:CompactGen}.  Since $A$ is abelian and totally disconnected, it is compact-by-discrete. Now Lemma~\ref{lem:compact-by-discrete-by-compact} ensures that $B$ has a compact open normal subgroup, say $V$. The quotient group $AV/V \cong A/A \cap V $ is finitely generated abelian. Moreover $AV$ is open and cocompact in $B$, hence of finite index. It follows that the discrete group $B/V$ is finitely generated and virtually abelian. Its LE-radical is thus finite. Since $V$ is compact, we have $\LF(B/V) = \LF(B)/V$. This implies that  $\LF(B)/V$ is finite, so that $\LF(B)$ is compact. Hence  $\LF(B)$ is a compact open normal subgroup of $G$.
\end{proof}

We can now complete the proof of the theorem.

\begin{proof}[Proof of Theorem~\ref{thm:Noetherian}]
Let $G$ be a locally compact group with a closed cocompact normal subgroup $N$ which is solvable Noetherian.
The properties of solvability and Noetherianity are both preserved under passing to the closure of the image under a continuous homomorphism of locally compact groups. Hence the closure of the image of $N$ in $G/G^\circ$ is a closed cocompact normal subgroup which is solvable Noetherian.  By Theorem~\ref{theo:reduction}, we may thus assume that  $G$ is totally disconnected. Proposition~\ref{prop:SolubleNoether} then guarantees that $G$ is \{compact-by-discrete\}-by-compact,  and the conclusion follows from Corollary~\ref{cor:compact-by-discrete}.
\end{proof}

We record the following result, although it won't be used in the sequel.
\begin{prop}\label{prop:Polycyclic}
Let $G$ be a Noetherian locally compact group. If $G$ is solvable, then every closed subgroup is compactly generated.
\end{prop}

\begin{proof}
By \cite{HofmannNeeb}, every closed subgroup of an almost connected solvable locally compact group is compactly generated.  By Proposition~\ref{prop:SolubleNoether}, the group $G$ has an almost connected open normal subgroup $N$. Thus the open normal subgroup $N$ and the discrete quotient $G/N$ both have the property that every closed subgroup is compactly generated. Thus the property is inherited by $G$, as desired.
\end{proof}

The following consequence is immediate (compare \cite{Guivarch}*{Th\'eor\`eme III.1}).

\begin{cor}\label{cor:Chain}
Let $G$ be a Noetherian solvable locally compact group. Then $G$ has a subnormal chain $G=G_0\rhd G_1\rhd G_2,\rhd\ldots\rhd G_n=\{1\}$ such that $G_{i-1}/G_{i}$ is either compact abelian or isomorphic to $\ZZ$ or $\RR$, for all $0<i\leq  n$. \qed
\end{cor}

\subsection{Locally elliptic groups with the bounded index property}

In all the classes of groups we have treated so far, the proof proceeded in two steps: first invoking Theorem~\ref{theo:reduction} to reduce to totally disconnected groups, and then observe that the totally disconnected members of the class of groups in question were \{compact-by-discrete\}-by-compact, so that the required conclusion followed via Corollary~\ref{cor:compact-by-discrete}. We shall now establish property (M) for a particular class of locally elliptic groups; this will happen to be a crucial tool in establishing the subsequent results in the rest of this chapter.

\begin{prop}\label{prop:criterion}
Let $O$ be a locally compact group which is the union of a countable ascending chain of compact open subgroups $O_1 \leq O_2 \leq O_3 \leq \dots$. 
Let $H<O$ be a closed subgroup. 
Let $\vol_O$ be a Haar measure on $O$ and $\vol_H$ be a Haar measure on $H$. 
Set  $H_n=H\cap O_n$ and 
$\gamma_n=\vol_O(O_n)/\vol_H(H_n)$.
Then $(\gamma_n)$ is a non-decreasing sequence, converging to the covolume of $H$ in $O$ in the standard normalization (see \ref{sec:SE}).  In particular, $H$ is of finite covolume if and only if $(\gamma_n)$ is bounded.
Moreover, $H$ is cocompact if and only if $(\gamma_n)$ is asymptotically constant.
\end{prop}

\begin{proof}
Straightforward.
Observe that $O$ and all of its closed subgroups are unimodular. In particular,  if $H$ is cocompact then it is of finite covolume.
\end{proof}

\begin{prop}\label{prop:UniformBound}
Let $O$ be a locally compact group which is the union of a countable ascending chain of compact open subgroups $O_1 \leq O_2 \leq O_3 \leq \dots$. If $\sup_n [O_{n+1} : O_n] < \infty$, then $O$ has property $(M)$.
\end{prop}

\begin{proof}
Let $H<O$ be a closed subgroup of finite covolume. We let $\vol_O$ and $\vol_H$ denote Haar measures on $O$ and $H$ respectively, and   $\mu$ be an $O$-invariant measure on $O/H$ such that $\vol_O$,  $\vol_H$ and $\mu$ satisfy the standard normalization (see Section~\ref{sec:SE}).  Define  $H_n=H\cap O_n$ for all $n$, so that $H_n$ is a compact open subgroup of $H$ contained in $H_{n+1}$. Set 
$$
 \gamma_n=\mu(O_nH/H) = \vol_O(O_n)/\vol_H(H_n),
$$
where the second equality follows from the standard normalization of the measures. 
Set also
 $$ \alpha_n=[ H_{n+1} : H_n]~\text{and}~\beta_n=[ O_{n+1} : O_n]
$$
for all $n$ and note that $\alpha_n, \beta_n \in \NN$.
By Lemma~\ref{lem:serre}, we have
$$\vol_O(O_n) \vol_H(H_m) \mu(O_mH/H) = \vol_O(O_m) \vol_H(H_n) \mu(O_nH/H)$$
for all $m, n$, so that, putting $m = n+1$, we get
$$
\alpha_n \gamma_{n+1}=\beta_n\gamma_n
$$
for all $n$. 
In particular $\alpha_n\leq \beta_n$ since $\gamma_{n+1} \geq \gamma_n$.
By assumption we have an upper bound $q$ such that for all $n$, $\alpha_n\leq \beta_n\leq q$.

Denote $C=\gamma_1/q$.
We claim that if $\gamma_n<\gamma_{n+1}$ then $\gamma_{n+1}-\gamma_{n}\geq C$.
Indeed, Assume $\gamma_n<\gamma_{n+1}$.
Then $\beta_n>\alpha_n$.
Since these are integers, we get  $\beta_n\geq \alpha_n+1$ and conclude
\[\gamma_{n+1}-\gamma_n=\gamma_n(\frac{\beta_n}{\alpha_n}-1) \geq \gamma_n(\frac{\alpha_n+1}{\alpha_n}-1)
=\frac{\gamma_n}{\alpha_n}\geq \frac{\gamma_1}{q} =C. \]
Since $H$ is of finite covolume, by Proposition~\ref{prop:criterion}, $(\gamma_n)$ converges, so $(\gamma_{n+1}-\gamma_n)$ tends to 0.
It follows that $(\gamma_n)$ stabilizes. Again, by Proposition~\ref{prop:criterion}, we conclude that $H$ is cocompact.
\end{proof}

We remark that an extension of a group satisfying the hypotheses of Proposition~\ref{prop:UniformBound} by a compact group need not have  property (M): indeed, if one chooses all the finite fields in Theorem~\ref{theo:example} to have the same characteristic~$p$, then the group $\Lambda$ from that example is a discrete elementary abelian $p$ group, which therefore satisfies the condition of Proposition~\ref{prop:UniformBound}. However, Theorem~\ref{theo:example} provides an extension of $\Lambda$ by a compact group which fails to  have  property (M).

That example also shows that, although the class of locally elliptic groups is stable under group extensions, the class of locally elliptic groups with the bounded index property is not preserved by group extensions. However, we do have the following result for locally $p$-elliptic groups.

\begin{lem}\label{lem:Locally-p-elliptic->bddIndex}
Let $p$ be a prime. Every $\sigma$-compact locally $p$-elliptic \tdlc group has the bounded index property. In particular it has property (M).
\end{lem}

\begin{proof}
Let $O$ be a $\sigma$-compact locally $p$-elliptic \tdlc group.  Let $O_1 \leq O$ be a compact open subgroup. Thus $O_1$ is of countable index since $O$ is $\sigma$-compact.  Proposition~\ref{prop:Locally-p-elliptic} implies that the existence of a chain $O_1 \subset O_2 \subset \dots$ consisting of open pro-$p$ subgroups whose union is the whole group $O$. Since the $O_n$ are all pro-$p$, the chain $O_1 \subset O_2 \subset \dots$ may be refined in such a way that  $[O_{n+1} : O_n] = p$ for all $n$. Thus $O$ has the bounded index property, and we conclude using Proposition~\ref{prop:UniformBound}.
\end{proof}

We next provide another set of conditions under which the hypotheses of Proposition~\ref{prop:UniformBound} are satisfied.

\begin{lem}\label{lem:BddIndex}
Let $G$ be a \tdlc group. Assume that $G$ has a closed normal subgroup $A$ and a compact open subgroup $U$ in $G$ satisfying the following conditions:
\begin{enumerate}[(a)]
\item  $G/A$ is compactly  generated and has a compact  open normal subgroup.

\item $U \cap A$ is   normal in $A$.

\end{enumerate}
Let $B$ be the normal closure of $U \cap A$ in $G$. Then $O= BU$ has property (M). In particular, if $H$ is a closed cofinite subgroup of $G$, then $H \cap O$ is cocompact in $O$ and $H$ is cocompact in $\overline{HB}$.
\end{lem}

\begin{proof}
Since $G/A$ has a compact open normal subgroup, it follows that $G$ has an open normal subgroup, say $P$, containing $A$ as a cocompact subgroup.

By (a) the quotient  $G/P$ is finitely generated. We may thus find  a finite set of elements $S \subset G$ which, together with $P $, generates $G$.
Let $\Gamma$ be the subgroup of $G$ generated by $S$, let $\ell$ be the word length of $\Gamma$ with respect to the finite generating set $S$ and let $1= g_0 , g_1, g_2, \dots$ be an enumeration of the elements of $\Gamma$ such that $\ell(g_i) \leq \ell(g_{i+1})$ for all $i$.

The group  $V= U \cap A$ is normal in $U$ since $A$ is normal in $G$; it follows that the normalizer $\norma_P(V)$ is open in $P$. On the other hand, the normalizer $\norma_P(V)$ contains $A$ by (b), and is thus cocompact in $P$. Therefore $\norma_P(V)$ is of finite index in $P$. In other words, the $P$-conjugacy class of $V$ is finite.

Let $V=V_1, V_2, \dots, V_n$ be all the $P$-conjugates of $V$ and set $A_0 = V_1 \cdot V_2 \dots V_n$. Thus $A_0$   is a compact normal subgroup of $P$ which contains $V$ and  is thus open in $A$.

For each $n>0$, let $A_n \leq A$ be the subgroup of $A$ generated by $A_0 \cup g_1 A_0 g_1\inv \cup \dots \cup g_n A_0 g_n\inv$. Since $P$ is normal in $G$, it follows that for each $n \geq 0$, the group  $g_n A_0 g_n\inv$ is a compact normal subgroup of $P$ contained in $A$. In particular,   we have $A_n = A_{n-1} \cdot g_n A_0 g_n\inv$ for all $n>0$. Given $n>0$, there is some $m\leq n-1$ and $s \in S$ such that $g_n = g_m s$. We deduce successively
$$
\begin{array}{rcl}
[A_n : A_{n-1}]  & = &  [g_n A_0 g_n\inv : A_{n-1} \cap  g_n A_0 g_n\inv] \\
& \leq & [g_n A_0 g_n\inv : g_{m} A_0 g_{m}\inv \cap g_n A_0 g_n\inv] \\
& = & [s A_0 s\inv : A_0 \cap s A_0 s \inv ] \\
& \leq & \sup_{t \in S} [t A_0 t\inv : A_0 \cap t A_0 t \inv ]
\end{array}
$$

Since $A_n$ is normal in $P$ for all $n$, it follows that     $\bigcup_{n\geq 0} A_n$ is normal in $G$ and open in $A$, hence closed. Therefore, the group $ \bigcup_{n\geq 0} A_n$ coincides with the normal closure of $V$ in $G$, which is denoted by $B$.

We now set $Q_n = A_n \cdot (P \cap U)$ for each $n\geq 0$  so that the group  $Q = \bigcup_{n\geq 0} Q_n$ coincides with $A \cdot (P \cap U)$. Moreover, for all $n>0$, we have $Q_n = A_n \cdot (P \cap U)$, so that  
$$
[Q_n : Q_{n-1}] \leq  [A_n : A_{n-1}].
$$
We have shown above that the latter is bounded above independently of $n$ by $ \sup_{t \in S} [t A_0 t\inv : A_0 \cap t A_0 t \inv ]$. It follows from Proposition~\ref{prop:UniformBound} that $Q$ has property (M). Since $P \cap U$ is of finite index in $U$, it follows that $Q = B(P\cap U)$ is of finite index in $O = BU$, so that $O$ has property (M) as well. 

Given a cofinite subgroup $H \leq G$, we deduce that $H \cap O$ is cocompact in $O$, and this in turn implies that $H$ is cocompact in $\overline{HB}$ by Lemma~\ref{lem:HUGA}.
\end{proof}

\subsection{Metabelian-by-compact groups}

The examples of groups without property (M) provided by Theorem~\ref{theo:example} are metabelian, but not compactly generated. The following result shows that for metabelian groups, compact generation implies property (M).

\begin{thm}\label{thm:metabelian}
A compactly generated metabelian-by-compact group has property (M).
\end{thm}

\begin{proof}
Let $G$ be a compactly generated locally compact group with a closed  normal subgroups $N \leq M$ such that $N$ is abelian, $M/N$ is abelian and $G/M$ is compact. We must prove that $G$ has property (M).

We first observe that $G$ is amenable and that the group $G/G^\circ$ is also metabelian-by-compact. By Theorem~\ref{theo:reduction} we may assume that $G$ is totally disconnected.

Let $V$ be a compact relatively open subgroup of $N$, and  $A$ denote the (abstract) normal closure of $V$ in $G$. Since $V$ is relatively open in $N$, so is $A$; in particular $A$ is closed in $G$. Moreover $V$ is normal in $A$ since $A$ is abelian. The quotient $G/A$ is compactly generated. We claim that it possesses a compact open normal subgroup.

In order to establish the claim, observe that $M/N$ is a compactly generated abelian totally disconnected group. Dividing out a compact open subgroup, we obtain a finitely generated abelian group, whose torsion subgroup is thus finite. Therefore $M/N$ has a unique maximal compact open subgroup, and the corresponding quotient is torsion-free. It follows that $M/N$  splits as a direct product of a compact open subgroup and a discrete torsion-free subgroup. Since $N/A$ is discrete, we deduce that the compactly generated group $M/A$ has a discrete cocompact normal subgroup. By Lemma~\ref{lem:DiscreteNormal}, this implies that $M/A$ has a compact open normal subgroup. Therefore $G/A$ is \{compact-by-discrete\}-by-compact. The claim then follows from Lemma~\ref{lem:compact-by-discrete-by-compact}.

Choose a compact open subgroup $U$ of $G$ such that $V = U \cap A$. The claim implies that all hypotheses of Lemma~\ref{lem:BddIndex} are satisfied.  Since $A$ is the normal closure of $V$, Lemma~\ref{lem:BddIndex} ensures that a cofinite subgroup $H$ of $G$ is cocompact in $\overline {HA}$. Since $G/A$ has a compact open normal subgroup, it also has property (M) by Lemma~\ref{lem:CompactKernel}, so that $\overline {HA}$ is cocompact in $G$. Hence $G/H$ is compact.
\end{proof}

\subsection{Nilpotent-by-nilpotent groups}

We are now ready to prove Theorem~\ref{theo:nil-by-nil}, which we reformulate as follows.

\begin{thm}\label{thm:nil-by-nil}
Compactly generated nilpotent-by-nilpotent locally compact groups $G$  have   property (M).
\end{thm}

\begin{proof}
Let $G$ be a compactly generated locally compact group with a closed  normal  nilpotent subgroup $N$ such that   $G/N$ is nilpotent. We must prove that $G$ has property (M).  By Theorem~\ref{theo:reduction} we may assume that $G$ is totally disconnected. We let $H \leq G$ be a cofinite subgroup. The goal is to show that $H$ is cocompact.

We shall work by induction on the nilpotency class of $N$. We start by  assuming  that $N$ is abelian.

Let $U$ be a compact open subgroup of $G$. Set $V = U \cap N$ and let $B$ be the normal closure of $V$ in $G$. Since $V$ is relatively open in $N$, so is $B$, so that $B$ is closed. By Proposition~\ref{prop:nil-Willis}, the quotient $G/N$ has a compact open normal subgroup. By Lemma~\ref{lem:BddIndex}, this implies that $H$ is cocompact in  $\overline{HB}$. Therefore, it suffices to show the cocompactness of $\overline{HB}$ in $G$. To this end, we may replace $G$, $N$ and $H$ by $G/B$, $N/B$ by $\overline{HB}/B$. In view of that reduction, we assume henceforth that $U \cap N$ is trivial; in particular $N$ is discrete.

Since $G/N$ has a basis of identity neighborhoods consisting of compact open normal subgroups by Proposition~\ref{prop:nil-Willis}, we may replace $U$ by an open subgroup so as to ensure that $UN$ is normal in $G$. Moreover, the fact that $G$ is compactly generated and that $G/N$ is compactly presented (by Proposition~\ref{prop:Nilpotent->CptlyPres}) implies (by Proposition~\ref{prop:CompactlyPresented:DiscreteExtension}) that $N$ is finitely generated as a normal subgroup. Since $U$ is compact and $N$ is discrete, we may find a $U$-invariant finite subset $S$ of $N$ which generates $N$ as a normal subgroup. Let $N_0 = \langle S \rangle$. Thus $N_0$ is normalized by $U$, hence by $UN$. Since $UN$ is normal in $G$, we see that for any $g \in G$, the group $g N_0 g^{-1}$ is normal in $UN$, hence normalized by $U$.
Since $N_0$ is a finitely generated abelian group, there is an upper bound $q$ on the order of finite subgroups of $\Aut(N_0)$. It follows that for any $g \in G$, the quotient $U/C_U(gN_0 g^{-1})$ is of order~$\leq q$. Therefore the quotient of $U$ by $\bigcap_{g \in G} C_U(gN_0 g^{-1})$ is of exponent~$\leq q$. Since $N$ is the normal closure of $N_0$ in $G$, we have $C_U(N) = \bigcap_{g \in G} C_U(gN_0 g^{-1})$. Thus $U/C_U(N)$ is of exponent~$\leq q$.

Since $N$  and $NU$ are both closed normal subgroups of $G$, the centralizer $C_{NU}(N)$ is also a closed normal subgroup of $G$. Moreover, since $N$ is abelian, we have $C_{NU}(N) = N C_U(N)$. Observe that $G/N C_U(N)$ is a quotient of $G/N$, so that it has a compact open normal subgroup (namely the image of $U$). Moreover $C_U(N)$ is a compact relatively open normal subgroup of $N C_U(N)$. We may thus invoke  Lemma~\ref{lem:BddIndex}, which ensures that  the group $H$ is thus cocompact in $\overline{HC}$, where $C \leq N C_U(N)$ denotes the normal closure of $C_U(N)$ in $G$. Therefore, it suffices to show the cocompactness of $\overline{HC}$ in $G$. To this end, we may replace $G$, $N$ and $H$ by $G/C$, $\overline{NC}/C$ by $\overline{HC}/C$. In view of that reduction, we assume henceforth that $U$ is of exponent~$\leq q$.

(Note however that in the last reduction, the group $N$ may have lost its discreteness a priori. We could perform one more application of Lemma~\ref{lem:BddIndex} to divide out the normal closure of $U\cap N$ so as to restore the discreteness of $N$ without affecting the boundedness of the exponent of $U$. This is however not necessary for the rest of the argument, so we will continue the discussion without assuming that $N$ is discrete.)

Since $G/N$ is nilpotent, it has property (M) by Proposition~\ref{prop:nilpotent}. Therefore $\overline{HN}$ is cocompact in $G$ and we may assume that $G = \overline{HN}$. Now the intersection $H\cap N$ is normal in $H$ (because $N$ is normal in $G$) and commutes with $N$ (because $N$ is abelian), it is thus normal in $G$ since $G = \overline{HN}$. We may thus pass to the quotient $G/H\cap N$, and assume that $H \cap N$ is trivial.

After that reduction, we claim that $N$ is locally elliptic. Indeed, let $\Sigma \subset N$ be a compact subset; we need to show that the group $W = \overline{\langle \Sigma \rangle}$ is compact. Upon replacing $\Sigma$ by $U\Sigma U \cap N$, we may assume that $W$ is invariant under the $U$-action by conjugation. It follows that $UW$ is a compactly generated abelian-by-compact group. Therefore it is compact-by-\{virtually abelian\} in view of Lemma~\ref{lem:compact-by-discrete-by-compact}. Let $D$ be a discrete virtually abelian quotient of $UW$ by a compact open normal subgroup; hence $D$  is virtually $\mathbf Z^d$ for some $d \geq 0$. Now observe that  the intersection $H \cap UW$ is cofinite in $UW$ since $UW$ is open. Since $H \cap N$ is trivial, it follows that $H \cap UW$ injects in $UN/N$ which is of exponent~$\leq q$. Hence $H \cap UW$ is of exponent~$\leq q$. But the image of $H \cap UW$ in $D$ is of finite index, so that $D$ is a torsion group. Since a finitely generated virtually abelian torsion group is finite, it follows that $UW$ is compact, hence $W$ is compact as well.   This  confirms the claim.

Now we invoke again the compact presentability of $G/N$ (guaranteed by Proposition~\ref{prop:Nilpotent->CptlyPres}) to deduce that $N$ is compactly generated as a normal subgroup (see Proposition~\ref{prop:CompactlyPresented:DiscreteExtension}). Therefore we may argue as above to find a  $U$-invariant compact relatively open subgroup $W$ of $N$ such that $N$ is the normal closure of $W$ in $G$. We  then invoke Lemma~\ref{lem:BddIndex}, which ensures that $H$ is cocompact in $\overline{HN}=G$.
This finishes the proof in the case where $N$ is abelian.

\medskip
Now the induction can start, and we assume that the nilpotency class of $N$ is $n+1$. Let  $Z$ denote the center of $N$, which is closed and normal in $G$. Let also $H$ be a closed subgroup of finite covolume in $G$. The group $\overline{HZ}/Z$ is a closed subgroup of finite covolume in $G/Z$. By the induction hypothesis, it is cocompact. We may thus replace $G$ by  $\overline{HZ}$ and assume that $HZ$ is dense.

We next consider the intersection $H \cap N$. It is normalized by $H$ (since $N$ is normal in $G$) and by $Z$ (since $Z$ is central in $N$). Since $H \cap N$ is closed and $HZ$ is dense in $G$, we deduce that $H \cap N$ is normal in $G$. We may therefore replace $G$, $H$ and $N$ by their canonical images in the quotient group $G/H \cap N$, which are closed. In particular $Z$ is replaced by the closure of its image in $G/H \cap N$, and it remains true that $Z$ is an abelian closed normal subgroup and that $HZ$ is dense in $G$. After that reduction, the intersection $H \cap N$ is trivial, so that $H$ injects in the quotient $G/N$, which is compactly generated and nilpotent.
Since $H$ maps onto a dense subgroup of  $G/Z$, we infer that $G/Z$ is nilpotent, so that $G$ is abelian-by-nilpotent. The cocompactness of $H$ then follows from the base case of the induction.
\end{proof}

\subsection{On continuous finite-dimensional representations of $p$-adic Lie groups}\label{sec:p-adicLie}

The following result does not seem to appear in the literature.

\begin{prop}\label{prop:Lie}
Let $p$ be a prime and $G$ be a $p$-adic Lie group. For any continuous representation  $\varphi \colon G \to \GL_d(k)$  over a locally compact field $k$, there exists an open normal subgroup $M \leq G$ such that $\varphi(\overline{[M, M]})$ is closed in $\GL_d(k)$.
\end{prop}

We shall use the following lemma, valid in arbitrary characteristic, where $\overline K^Z$ denotes the Zariski closure of a set $K$.

\begin{lem}\label{lem:Zclosure}
Let $G$ be a \tdlc group and $\varphi \colon G \to \GL_d(F)$ be a   representation over an arbitrary field $F$.

For each closed normal subgroup $H \leq G$ there exists a closed normal subgroup $M \leq H$  relatively open in $H$ and such that
$$\overline{\varphi(M)}^Z = \overline{\varphi(H \cap O)}^Z$$
for any open subgroup $O \leq G$.
\end{lem}

\begin{proof}
Since any descending chain of Zariski closed subsets terminates, there exists a compact open subgroup $U \leq G$ such that the Zariski-closure $\overline{\varphi(H \cap U)}^Z$ is minimal among all Zariski closures of images of relatively open subgroups of $H$.
Let now $g \in G$ and consider the compact open subgroup $V = g^{-1} U g \cap U$. We have
$$\varphi(g) \overline{\varphi(H \cap U)}^Z \varphi(g)^{-1} = \varphi(g) \overline{\varphi(H \cap V)}^Z \varphi(g)^{-1} =  \overline{\varphi(H \cap g Vg^{-1} )}^Z = \overline{\varphi(H \cap U)}^Z.$$
Thus $ \overline{\varphi(H \cap U)}^Z$ is normalized by $\varphi(G)$, so that $M = H \cap \varphi^{-1}(\overline{\varphi(H \cap U)}^Z)$ is a   normal subgroup of $G$. Moreover  $M$  contains $H \cap U$ and is thus relatively open in $H$, hence closed in $G$. It enjoys the required property by construction.
\end{proof}

\begin{proof}[Proof of Proposition~\ref{prop:Lie}]
Let $M \leq G$ be the open normal subgroup of $G$ obtained by applying Lemma~\ref{lem:Zclosure} to $H = G$.
Let also  $\mathfrak g$ denote the image of the $p$-adic Lie algebra of $G$ induced by $\varphi$. By \cite{Bourbaki}*{Ch.~III, \S9, Prop.~6}, there exists a compact open subgroup $U \leq M$ such that the Lie algebra of the derived group of $\varphi(U)$ coincides with the derived Lie algebra $[\mathfrak g, \mathfrak g]$. A result of Chevalley (see \cite{Borel}*{Ch.~II, Cor.~7.9}) now ensures that the derived Lie algebra $[\mathfrak g, \mathfrak g]$ is the Lie algebra of a unique   connected Zariski-closed subgroup $H \leq \GL_d(k)$.  Since $H$ is algebraic, its Lie algebra as an algebraic group coincides with its Lie algebra as a $p$-adic analytic group. It follows that the compact group $\varphi([U, U])$, whose Lie algebra is  $[\mathfrak g, \mathfrak g]$, must be relatively open in $H$. Moreover we have $\overline{\varphi([U, U])}^Z = H$. Since $\overline{\varphi(U)}^Z = \overline{\varphi(M)}^Z$, we infer that $H$ is normalized by $\overline{\varphi(M)}$ and that $\overline{\varphi(M)}H/H$ is abelian. In particular we have $\varphi(\overline{[M, M]})\leq \overline{\varphi([M, M])} \leq H$.

Since $\varphi([U, U])$ is relatively open in $H$ and $U \leq M$, it follows that $\varphi(\overline{[M, M]})$ is relatively open in $H$, hence closed.
\end{proof}

\subsection{On compact linear groups}\label{sec:Pink}

The structure of compact subgroups of linear algebraic groups over local fields has been described in a seminal paper due to R.~Pink \cite{Pink}. We shall use the following statement that follows from his work.

\begin{thm}\label{thm:Pink}
Let $k$ be a non-Archimedean local field  and $Q$ be a compact subgroup of $\GL_d(k)$. Then  there exists closed subgroups $U \leq R \leq S \leq Q_1 \leq Q$ such that:
\begin{enumerate}[(i)]
\item $Q_1$ is pro-$p$, where $p$ is the residue characteristic of $k$. Moreover $Q_1$ is open normal in $Q$, and $U, R$ and $S$ are normal in $Q_1$.

\item $U$ is nilpotent; moreover, if the characteristic of $k$ is $p>0$, then every finitely generated subgroup of $U$ is a finite $p$-group of bounded exponent.

\item $[Q, R] \leq U$.

\item $S/R$ is a finite direct product of non-virtually abelian hereditarily just-infinite groups that are   isomorphic to compact open subgroups in simple algebraic groups over local fields.

\item $Q_1/S$ is abelian of finite exponent.

\end{enumerate}
\end{thm}

We recall that a profinite group is called \textbf{just-infinite} if every non-trivial closed normal subgroup is of finite index, and \textbf{hereditarily just-infinite} if every open subgroup is just-infinite.

\begin{proof}
There exists an open normal subgroup $Q_1 \leq Q$ which is pro-$p$, and whose Zariski-closure $\mathbf Q$ is Zariski-connected. Let $\mathbf R$ be the solvable radical of $\mathbf Q$, and $\mathbf U$ be the unipotent radical of $\mathbf R$. We set
$$R_1 =  \mathbf R \cap Q_1, \ U_1=  \mathbf U \cap Q_1, \ R = \overline{R_1} \text{ and } U = \overline{U_1}.$$
Then Assertion (i) holds by construction.

Since $\mathbf U$ is unipotent, it is nilpotent, and so are thus $U_1$ and $U$. If moreover the characteristic of $k$ is $p>0$, then every finitely generated subgroup of $\mathbf U$ is a finite $p$-group of bounded exponent. That property is thus inherited by $U_1$. The closure $\overline{U_1}$ is thus a pro-$p$ and all of whose elements are torsion (of bounded exponent). Every finitely generated subgroup of $\overline{U_1}$ is thus linear and $p$-torsion, hence a finite $p$-group (of bounded exponent). This proves (ii).

Since $\mathbf R/\mathbf U$ is a connected abelian normal subgroup of the connected reductive group $\mathbf Q/\mathbf U$, it is contained in the center. Thus we have $[\mathbf Q, \mathbf R] \leq \mathbf U$. This implies $[Q_1, R_1]\leq U_1$, whence $[Q_1, R] \leq U$. This proves (iii).

The image of $Q_1$ in the semi-simple group $\mathbf Q/\mathbf R$ is Zariski-dense. Therefore, the image of the solvable normal subgroup $R$ of $Q_1$ is finite, and the quotient $Q_1/R$ has a maximal solvable normal subgroup which must be finite. Now the existence of a normal subgroup $S \leq Q_1$ containing $R$ follows from Pink's work \cite{Pink}*{Cor.~0.5} and its extension \cite{CapStu}*{Th.~4.13}.
\end{proof}

\subsection{Structure of amenable linear groups}\label{sec:linearStructure}

Let $k$ be a locally compact field.
A locally compact group $G$ is called \textbf{linear over the field $k$}  (or simply \textbf{linear}) if there is an integer $d >0$, a locally compact field $k$ and a continuous injective homomorphism $G \to \GL_d(k)$. The following result describes the algebraic structure of amenable linear locally compact groups.

\begin{thm}\label{thm:StructureAmenableLinear}
Let $G$ be an amenable locally compact group. Assume that $G$ is linear over a locally compact field $k$.

\begin{enumerate}[(i)]
\item If $k$ is connected, then $G$ is Lie group, $G^\circ$ is solvable-by-compact and $G/G^\circ$ is virtually solvable.

\item If $k$ is disconnected, then there is a prime $p$ such that the $p$-radical $\LFp(G)$ is open in $G$.

\end{enumerate}

\end{thm}

We emphasize that the definition of linearity does not require the image of the faithful linear representation to be closed. The first step in the proof of Theorem~\ref{thm:StructureAmenableLinear} is actually to check the statement for amenable closed subgroups of $\GL_d(k)$. In the non-Archimedean case, this can be proved with the aid of the following statement, due to Guivarc'h--R\'emy \cite{GuiRem}.

\begin{thm}\label{thm:GuivarchRemy}
Let $k$ be a non-Archimedean local field of residue characteristic $p$ and $A \leq \SL_d(k)$ be a closed amenable subgroup. Then $\LFp(A)$ is open, and the discrete quotient $A/\LFp(A)$ is virtually abelian.

More precisely $A$ has an open normal subgroup of finite index $A^0$ contained in a parabolic subgroup $P$ of $\SL_d(k)$, and the closure of the image of $A^0$ in the Levi factor of $P$ is compact-by-\{discrete virtually abelian\}.
\end{thm}

\begin{proof}
It follows from \cite{GuiRem}*{Th.~33, Cor.~34 and the definitions from \S1.2} that $A$ has an open normal subgroup $A^0$ satisfying the required properties. In particular $A^0$ is contained in a closed subgroup $B$ of $P$ which contains the unipotent radical of $P$, and whose image in the Levi factor is  compact-by-\{discrete abelian\}. Since the unipotent radical of $P$ is locally $p$-elliptic, and since every compact subgroup of $\SL_d(k)$ is virtually pro-$p$, we deduce that $\LFp(B)$ is open, and that $B/\LFp(B)$ is virtually abelian. Therefore $A^0 \cap \LFp(B) \leq \LFp(A^0)$ is open in $A^0$, and the quotient $A^0/\LFp(A^0)$ is also virtually abelian. This finally implies that the same properties hold for $A$.
\end{proof}

The proof of Theorem~\ref{thm:StructureAmenableLinear} also requires the following subsidiary facts.

\begin{lem}\label{lem:discrete-by-compact}
Let $G$ be a \tdlc group. If the center of $G$ is cocompact, then the locally elliptic radical of $G$ is open.
\end{lem}

\begin{proof}
See \cite{CaMo_amenis}*{Th.~2.4}.
\end{proof}

\begin{lem}\label{lem:FGpro-p}
Let $G$ be a \tdlc group and $J = J_1 \times \dots \times J_m$ be a direct product of non-virtually abelian just-infinite pro-$p$ groups for some prime $p$. Let $\varphi \colon G \to J$ be a continuous homomorphism with dense image.

Then there exists a (possibly empty) subset $I \subset  \{1, \dots , m\}$ and a compact open subgroup $U \leq G$ such that $\varphi(U)$ is an open subgroup of the subproduct $\prod_{i \in I} J_i$.

In particular, if $J$ is just-infinite (i.e. if $m=1$), then either $\Ker(\varphi)$ is open or $\varphi$ maps each open subgroup of $G$ to an open subgroup of $J$.
\end{lem}
\begin{proof}
Set $H = G/\Ker(\varphi)$. The continuous homomorphism $H \to J$ induced by $\varphi$ is again denoted by $\varphi$.

Since every just-infinite pro-$p$ group is finitely generated, so is $J$, so that there exist $h_1, \dots, h_n \in H$ such that $\varphi(h_1), \dots , \varphi(h_n)$ generate a dense subgroup of $J$.

Since $H$ continuously injects in $J$, it is residually finite, hence residually discrete. Let $W \leq H$ be a compact open subgroup and $H_1 =  {\langle \{h_1, \dots, h_n\} \cup W\rangle}$. Thus $H_1$ is a compactly generated residually discrete open subgroup of $H$. By \cite{CaMo}*{Cor.~4.1}, the group $H_1$ has a compact open normal subgroup, say $V$. Now $\varphi(V)$ is a compact subgroup of $J$ which is normalized by $\varphi(H_1)$, hence by $\overline{\varphi(H_1)} = J$. Since $J$ is a direct product of non-virtually abelian  just-infinite groups, it follows that $\varphi(V)$ is a open subgroup of a subproduct $\prod_{i \in I} J_i$ for some $I \subset \{1, \dots, m\}$ (see \cite{CapStu}*{Lem.~3.6}). We finish by defining $U$ as a compact open subgroup of $G$ whose image in $H$ is contained in $V$.
\end{proof}

\begin{proof}[Proof of Theorem~\ref{thm:StructureAmenableLinear}]
(i) Follows from \cite{CaMo_amenis}*{Prop.~2.2} and Furstenberg's theorem ensuring that a connected Lie group is amenable if and only if it is solvable-by-compact.

\medskip \noindent
(ii) If $k$ is discrete, then so is $G$ and every subgroup is open.  We assume henceforth the $k$ is non-discrete. Thus $k$ is a non-Archimedean local field (see \cite{Weil}*{Ch.~1}). We let $p$ denote its residue characteristic.

Let $\varphi \colon G \to \GL_d(k)$ be a continuous injective homomorphism.
The group $\GL_d(k)$ embeds as a closed subgroup in $\SL_{d+1}(k)$. Therefore, Theorem~\ref{thm:GuivarchRemy} ensures that the $p$-radical of $\overline{\varphi(G)}$ is open. Therefore, upon replacing $G$ by an open normal subgroup, we may assume   that $\overline{\varphi(G)}$ is locally $p$-elliptic.

\bigskip
Assume now that the characteristic of $k$ is zero. Then $G$ is a $p$-adic Lie group. By Proposition~\ref{prop:Lie}, there is an open normal subgroup $M \leq G$ such that $\varphi(\overline{[M, M]})$ is closed. In particular it is locally $p$-elliptic, hence contained in $\LFp(M)$. Since $M/\overline{[M, M]}$ is abelian and $p$-adic Lie, it has an open normal pro-$p$ subgroup. This implies  by Proposition~\ref{prop:Locally-p-elliptic} that the $p$-radical $\LFp(M)$ is open in $M$, hence in $G$. Since $M$ is normal and $\LFp(M)
$ is characteristic, we have $\LFp(M) \leq \LFp(G)$, so that the latter is open.

\bigskip
We now assume that the characteristic of $k$ is positive (hence equal to $p$). We distinguish two cases, according to whether $Q = \overline{\varphi(G)}$ is compact or not.

Assume first that $Q$ is compact. We consider the closed subgroups $U \leq R \leq S \leq Q_1$ of $Q$ afforded by Theorem~\ref{thm:Pink}. 

Set $U_1 = \varphi^{-1}(U)$, $R_1 = \varphi^{-1}(R)$, $S_1 = \varphi^{-1}(S)$ and $G_1 = \varphi^{-1}(Q_1)$.

\begin{claim*}
$S_1/R_1$ is compact.
\end{claim*}

Since $Q_1/S$ is abelian, we have $[G_1, G_1] \leq S_1$, so that $\overline{\varphi([G_1, G_1])}R/R$ is a closed normal subgroup of $Q_1/R$ contained in $S/R$. By Theorem~\ref{thm:Pink}, the  group $S/R$ is a direct product $J_1 \times \dots \times J_m$ of finitely many hereditarily just-infinite groups. If a factor $J_i$   is not virtually contained in $\overline{\varphi([G_1, G_1])}R/R$, then it must intersect it trivially (since $J_i$ is just-infinite), and thus inject in the quotient $Q_1/\overline{\varphi([G_1, G_1])}$, which is abelian since $\varphi(G_1)$ is dense in $Q_1$. Since no factor $J_i$ is abelian, it follows that each $J_i$ is virtually contained in $\overline{\varphi([G_1, G_1])}R/R$.

We now consider the composite homomorphism
$$\overline{[G_1, G_1]}R_1/R_1 \to S/R \cong J_1 \times \dots \times J_m.$$
We have just seen that the closure of its image is an open subgroup of $J_1 \times \dots \times J_m$. Upon passing to an open subgroup of finite index, we may assume that the closure of its image is itself a direct product of non-virtually abelian just-infinite pro-$p$ groups. We then invoke Lemma~\ref{lem:FGpro-p} and denote by $I \subset \{1, \dots, m\}$ the afforded index set. If there exists $i \in \{1, \dots,m\} \setminus I$, then the kernel of the composite map  $\overline{[G_1, G_1]}R_1/R_1 \to S/R \cong J_1 \times \dots \times J_m \to J_i$ is open. Since $G$ is amenable, it then follows that  $J_i$ has an open subgroup possessing a dense subgroup which is amenable as an abstract group. Since $J_i$ is finitely generated pro-$p$, every dense subgroup contains a finitely generated dense subgroup. Since $J_i$ is linear (see Theorem~\ref{thm:Pink}), each finitely generated   amenable subgroup is virtually solvable by the Tits alternative \cite{TitsAlt}. We conclude that $J_i$ must be virtually solvable, which contradicts the fact that it is just-infinite but not virtually abelian. This shows that $I = \{1, \dots, m\}$. In view of  Lemma~\ref{lem:FGpro-p}, this means that the  homomorphism $\overline{[G_1, G_1]}R_1/R_1 \to S/R \cong J_1 \times \dots \times J_m$ maps open subgroups to open subgroups.

Let $V$ be a compact open subgroup of $S_1$. Thus $VR_1/R_1$ is open in $S_1/R_1$, so it has finite image in $S/R$ by the conclusion of the previous paragraph. Therefore it has finite index in $S_1/R_1$, so that  $S_1/R_1$ is compact. The claim stands proven.

\medskip

Since $Q_1/S$ is abelian, so is $G_1/S_1$. Since $G/G_1$ is finite, it follows that $G/S_1$ is abelian-by-finite. We may thus find an open normal subgroup $G_2$ of $G$ with  $S_1 \leq G_2 \leq G_1$ such that $G_2/S_1$ is compact. By the claim, the group $G_2/R_1$ is compact.

By Theorem~\ref{thm:Pink}, the group $U$ is locally $p$-elliptic as a discrete group. Therefore, so is $ U_1$. In particular $U_1 \leq \LFp(G)$.

By Theorem~\ref{thm:Pink}, we also have $[Q, R] \leq U$. Therefore $[G_2, R_1] \leq U_1$. Thus the quotient $G_2/U_1$ has a cocompact center. By Lemma~\ref{lem:discrete-by-compact} it has an open locally elliptic radical. Since $Q_1$ is pro-$p$, every compact open subgroup of $G_1$ is pro-$p$, so that the locally elliptic radical of $G_2/U_1$ coincides with its $p$-radical. This shows that $\LFp(G_2/U_1)$ is open in $G_2/U_1$. Since $U_1 \leq \LFp(G_2)$, we infer that $\LFp(G_2)$ is open in $G_2$. Since $G_2$ is open and normal in $G$, we have $\LFp(G_2) \leq \LFp(G)$ and $ \LFp(G)$ is indeed open. This concludes the proof in case $Q$ is compact.

\medskip
We   assume finally that $Q$ is non-compact. By Theorem~\ref{thm:GuivarchRemy}, the group $Q$ is contained in a parabolic subgroup $P$ of $\SL_{d+1}(k)$ so that the closure of the image of $Q$ in the Levi factor is compact. Since the unipotent radical of $P$ is locally $p$-elliptic as a discrete group, so is its pre-image under $\varphi$. Therefore, we may replace  the given representation $\varphi$ by its composite with the projection of $P$ onto its Levi factor. In this way, we are reduced to the case where the closure of the image of $\varphi$ is compact. That case has already been taken care of.
\end{proof}

\subsection{Amenable linear groups have property (M)}\label{sec:AmeanebleLinear}

In order to prove Theorem~\ref{theo:LinearAmenable}, we first establish the case where $n=1$, which can be restated as follows. 

\begin{thm}\label{thm:AmenableLinear->M}
Every amenable linear locally compact  group has property (M).
\end{thm}

\begin{proof}
Let $G$ be an amenable locally compact  group which is linear over a locally compact field $k$.

If $k$ is connected, then $G^\circ$ is open and solvable-by-compact in view of Theorem~\ref{thm:StructureAmenableLinear}. It then follows from Theorem~\ref{theo:reduction} that $G^\circ$ has (M). Therefore $G$ has (M) by Proposition~\ref{prop:discrete}(iii).

If $k$ is disconnected, then there is a prime $p$ such that $\LFp(G)$ is open in $G$ by Theorem~\ref{thm:StructureAmenableLinear}. It then follows from Lemma~\ref{lem:Locally-p-elliptic->bddIndex} that $\LFp(G)$ has (M), hence so does  $G$ by Proposition~\ref{prop:discrete}(iii).
\end{proof}

As a corollary, we recover (a compact extension of) Mostow's theorem.

\begin{cor}\label{cor:Mostow}
Every amenable Lie group has property (M).
\end{cor}

\begin{proof}
Let $G$ be an amenable Lie group. Then $G^\circ$ is open and it suffices by Proposition~\ref{prop:discrete}(iii) to prove the assertion for the connected group $G^\circ$. The kernel of the adjoint action of a connected Lie group on its Lie algebra is the center of the Lie group. Thus every connected Lie group is central-by-linear. Since property (M) is stable under central extensions (see Proposition~\ref{prop:CentralExt}), the required conclusion follows from Theorem~\ref{thm:AmenableLinear->M}.
\end{proof}

We are now able to complete the proof Theorem~\ref{theo:LinearAmenable}.

\begin{proof}[Proof of Theorem~\ref{theo:LinearAmenable}]
We proceed by induction on the number $n$ of direct factors of $G$. The base case of the induction is afforded by Theorem~\ref{thm:AmenableLinear->M}. We assume henceforth that $n>1$. If there is $i \in \{1, \dots, n\}$ such that $k_i$ is totally disconnnected, then $G_i$ is totally disconnected, hence $G$ has property (M) by Proposition~\ref{prop:DirectProd} in view of the induction hypothesis. Otherwise, each $k_i$ is connected. Since every connected locally compact field continuously embeds in $\mathbf C$, we infer that $G$ is  linear over $\mathbf C$. Hence  $G$ has property (M)  by invoking again Theorem~\ref{thm:AmenableLinear->M}. 
\end{proof}

\subsubsection{Solvable Lie groups --- Mostow's theorem revisited}\label{sec:Lie}

In this final subsection we sketch an alternative proof for Mostow's original theorem:

\begin{thm}[\cite{Mostow}]\label{thm:linear}
Every solvable Lie group has property $(M)$.
\end{thm}


\begin{proof}
Let $G$ be a compactly generated solvable Lie group and $H\le G$ a cofinite subgroup. Up to replacing $G$ and $H$ by finite index subgroups we may assume that the commutator $G' = [G,G]$ is nilpotent. We may argue by induction on the nilpotency degree of $G'$, where the base case when $G'$ is trivial follows from Proposition \ref{prop:nilpotent}.
Let $Z$ be the center of $G'$. By induction $(G/Z)/(\overline{HZ}/Z)$ is compact, hence we may assume that $HZ$ is dense in $G$. It follows that $F=H\cap G'$ is normal in $G$. Dividing by $F$ we are left to prove that $H/F$ is cocompact in $G/F$. Since $H/F$ is abelian, the result follows from the following lemma.
\end{proof}


\begin{lem}
Let $G$ be a  Lie group and $H$ a cofinite abelian subgroup. Then $H$ is cocompact.
\end{lem}

\begin{proof}[First proof --- assuming that $G$ is solvable]
Recall that compactly generated solvable Lie groups are Noetherian and note that every compactly generated abelian Lie group admits a lattice. Thus,
arguing by induction on the dimension and as in the proof above, we may reduce to the case that $G=T\ltimes U$ where $T$ is tori, $U$ a unipotent abelian group and the action of $T$ on $U$ is locally faithful (the kernel is discrete) and irreducible, $H\le G$ is a lattice, $H\cap U=\{1\}$ and $HU$ is dense in $G$. Moreover, unimodularity of $G$ implies in this case that $T$ is compact.
Under these conditions, $H$ cannot admit a sequence of nontrivial elements $h_n$ such that $h_n^{g_n}$ converges to the identity. Thus $H$ must be cocompact in $G$.
\end{proof}

\begin{proof}[Second proof]
By Lemma~\ref{lem:HUGA} and Proposition~\ref{prop:discrete}, we may replace $G$ by the identity component $G^\circ$ and $H$ by $H \cap G^\circ$. We assume henceforth that $G$ is connected.

Since $H$ is abelian, hence amenable, so is $G$. Therefore $G$ is amenable as well, and thus solvable-by-compact. We argue by induction on the solvability degree of the solvable radical of the Lie algebra $\mathfrak g = \mathrm{Lie}(G)$. In the case where the solvable radical is trivial, the group $G$ is
compact and the result is obvious.

We now assume that the radical of $\mathfrak g$ is non-trivial and let $A$ be a maximal non-trivial connected abelian normal subgroup of $G$. By induction, the closure of the image of $H$ in $G/A$ is cocompact. In other words the group $G_1 = \overline{AH}$ is cocompact in $G$. Therefore it suffices to show that $H$ is cocompact in $G_1$. Notice that $A$ and $H$ are both abelian, so that $G_1$ is metabelian.
For the same reason as before we may assume that $G_1$ is connected.

The group $H$ is abelian, and acts on the connected abelian Lie group $A$ by conjugation. Moreover, upon dividing out a compact normal subgroup (which we can do by Lemma~\ref{lem:CompactKernel}), we may assume that $A$ is isomorphic to $\RR^d$, where $d = \dim(A)$. Therefore there is a subgroup $B \leq A$ isomorphic to $\RR$ or $\RR^2$ invariant under $H$. We may further assume that $B$ is a minimal non-trivial closed connected $H$-invariant subgroup, so that the $H$-action on $B$ is irreducible.

We can now use induction on $d = \dim(A)$ to deduce that $\overline{BH}$ is cocompact in $G_1/B$. As before this reduces the problem to showing that $H$ is cocompact in $G_2 = \overline{BH}$. If $B \cong \RR$, then $G_2$ is virtually abelian, hence it has (M) and we are done. If $B \cong \RR^2$, then the image of $H$ in $\Aut(B) = \GL_2(\RR)$ is conjugate to $O(2)$ by irreducibility. Thus $G_2$ is a closed subgroup of $\RR^2 \rtimes O(2)$. Moreover the cofinite group $H$ intersects trivially the group $B \cong \RR^2$. Since $H$ is dense in $G_2/B$, we infer that $B$ contains an element acting as an irrational rotation. Since $B$ is abelian, it follows that the whole group $B$ acts as a group of rotation on $\RR^2$. Therefore $B$ must be a compact subgroup of $G_2$, which contradicts that $H$ is cofinite.
\end{proof}

%
%

Uri Bader uribader@gmail.com

Pierre-Emmanuel Caprace pierre-emmanuel.caprace@uclouvain.be

Tsachik Gelander tsachik.gelander@gmail.com

Shahar Mozes mozes@math.huji.ac.il

\end{document}